\documentclass[a4paper,12pt]{amsart}
\pdfoutput=1
\usepackage{amsmath}
\usepackage{amsfonts}
\usepackage{amssymb}
\usepackage{amscd}
\usepackage{t1enc}
\usepackage{t1enc}
   \usepackage{url}
\def\R{\mathbb{R}}

\def\cL{\mathcal{L}}

\def\al{\alpha}
\def\be{\beta}
\def\ga{\gamma}
\def\de{\delta}

\def\la{\lambda}

\def\si{\sigma}

\def\De{\Delta}

\def\Si{\Sigma}

\def\P{\mathrm{P}}

\newcommand{\der}{{\rm d}}

\newcommand{\newc}{\newcommand}

\newtheorem{theorem}{Theorem}[section]
\newtheorem{lemma}[theorem]{Lemma}
\newtheorem{proposition}[theorem]{Proposition}
\newtheorem{corollary}[theorem]{Corollary}
\theoremstyle{remark}
\newtheorem{remark}[theorem]{\rm\bf Remark}
\newtheorem*{remark*}{\rm\bf Remark}

\newtheorem{example}{Example}[theorem]

\usepackage{amssymb,stmaryrd}
\usepackage{amscd}

\newcommand{\nd}{\nabla}

\newc{\aR}{\mbox{\boldmath{$ R$}}}
\newc{\aS}{\mbox{\boldmath{$ S$}}}
\newc{\aDeR}{\mbox{\boldmath{$ U$}}_B{}^P{}_C{}^Q}
\newc{\aDe}{\mbox{\boldmath$ \Delta$}}
\newc{\aNd}{\mbox{\boldmath$ \nabla$}}

\newc{\aK}{\mbox{\boldmath{$ K$}}}
\newc{\aL}{\mbox{\boldmath{$ L$}}}

\newcommand{\on}[2]{\setbox0=\hbox{$#1$}\setbox1=\hbox{$#2$}%
            \dimen0=\wd0\advance\dimen0 by \wd1\divide\dimen0 by 2%
             \ifdim\wd0>\wd1$#1$\hskip-\dimen0$#2$\advance\dimen0 by -\wd1%
              \else$#2$\hskip-\dimen0$#1$\advance\dimen0 by -\wd0%
             \fi%
            \hskip\dimen0}

\makeatletter
\newcommand{\semidownbracefill}{$\m@th\braceld\leaders\vrule\hfill\braceru
  \bracelu\leaders\vrule\hfill\arrowhead$}



\def\sideremark#1{\ifvmode\leavevmode\fi\vadjust{\vbox to0pt{\vss
 \hbox to 0pt{\hskip\hsize\hskip1em
 \vbox{\hsize3cm\tiny\raggedright\pretolerance10000
 \noindent #1\hfill}\hss}\vbox to8pt{\vfil}\vss}}}%
                        
\author{Pawe\l~ Nurowski} 
\address{Centrum Fizyki Teoretycznej,
Polska Akademia Nauk, Al.\ Lotnik\'ow 32/46, 02-668 Warszawa, Poland}
\email{nurowski@cft.edu.pl}

\author{Matthew Randall}
 \address{Leibniz Universit\"at Hannover, Institut f\"ur Differentialgeometrie, Welfengarten 1, D-30167 Hannover}
\email{matthew.randall@math.uni-hannover.de}
\thanks{This research was supported by the Polish National Science Center (NCN) via grant DEC-2013/09/B/ST1/01799.}
\title{Generalised Ricci Solitons}

\begin{document}

\begin{abstract}
We introduce a class of overdetermined systems of partial differential equations of finite type on (pseudo)-Riemannian manifolds that we call the generalised Ricci soliton equations. These equations depend on three real parameters. For special values of the parameters they specialise to various important classes of equations in differential geometry. Among them there are: the Ricci soliton equations, the vacuum near-horizon geometry equations in general relativity, special cases of Einstein-Weyl equations and their projective counterparts, equations for homotheties and Killing's equation. 

We also prolong the generalised Ricci soliton equations and, by computing
differential constraints, we find a number of necessary conditions for a (pseudo)-Riemannian manifold $(M, g)$ to locally admit non-trivial solutions to the generalised Ricci soliton equations in dimensions 2 and 3. 

The paper provides also a collection of explicit examples of generalised Ricci solitons in dimensions 2 and 3 (in some cases).

\end{abstract}

\maketitle

\tableofcontents

\pagestyle{myheadings}
\markboth{Nurowski and Randall}{Generalised Ricci Solitons}

\section{Introduction}

Let $(M^n,g)$ be an oriented smooth manifold of dimension $n$ with (pseudo)-Riemannian metric $g$. Let us consider the system of equations
\begin{equation}\label{gs}
\cL_{X}g=-2c_1 X^\flat \odot X^\flat +2c_2 Ric + 2\la g
\end{equation}
on a vector field $X$. Here $\cL_{X}g$ is the Lie derivative of the metric $g$ with respect to $X$, $X^\flat$ is a 1-form such that $\langle X, X^\flat \rangle=g(X,X)$, {\it Ric} is the Ricci tensor of $g$ and $c_1$, $c_2$ and $\la$ are arbitrary real constants. 
In abstract index notation equation (\ref{gs}) can be rewritten as 
\begin{align}\label{grs0}
\nd_{(a}X_{b)}+c_1 X_a X_b-c_2 R_{ab}= \la g_{ab}
\end{align}
where the 1-form $X_a$, which is not necessarily closed, is given by $X_a=g_{ab}X^b$. Here $\nd$ is the Levi-Civita connection for the metric $g$. 
Let us call (\ref{grs0}) the generalised Ricci soliton equations. We note that there is not a unique way of assigning a name to this class of equations, and our choice is a matter of convenience (some other names were proposed such as GRicci solitons and grolitons). A pair $(g,X)$ is called a generalised Ricci soliton if (\ref{grs0}) is satisfied for some non-zero $X_a$ and metric $g_{ab}$.

If $c_1 \neq 0$, we can redefine $\tilde X_a=c_1 X_a$, so that (\ref{grs0}) is satisfied iff
  \begin{align}\label{grs1}
   \nd_{(a}\tilde X_{b)}+\tilde X_a \tilde X_b-\tilde c_2 R_{ab}=\tilde \la g_{ab}
   \end{align}
   holds, where $\tilde c_2=c_1c_2$ and $\tilde \la=c_1 \la$.
Hence we can redefine constants and study the equations
  \begin{align}\label{grs2}
   \nd_{(a}X_{b)}+X_a X_b-c_2 R_{ab}=\la g_{ab}
   \end{align}
   if $c_1 \neq 0$, and
\begin{align}\label{grs3}
   \nd_{(a}X_{b)}-c_2 R_{ab}=\la g_{ab}
   \end{align}
   if $c_1=0$.

In the case where $c_1=0$, if $c_2 \neq 0$ we can further rescale $X_a$ by $\tilde X_a=-\frac{1}{c_2} X_a$, to set $c_2=-1$. In this case (\ref{grs3}) is satisfied iff
  \begin{align}\label{grs4}
   \nd_{(a}\tilde X_{b)}+ R_{ab}=\tilde \la g_{ab}
   \end{align}
   holds, where $\tilde \la=-\frac{\la}{c_2}$.
This is the classical Ricci soliton equation. If $c_1=0$ and $c_2=0$, equation (\ref{grs3}) reduces to the equation for homotheties. 

\subsection{Motivation}
The generalised Ricci soliton equations (\ref{grs0}) contain in many cases equations of importance and interest in differential geometry.

The $c_1=0$, $c_2=-1$ case is the Ricci solitons with constant $\la$, which is called steady if $\la=0$, expanding if $\la<0$, and shrinking if $\la>0$; see \cite{ricsol} and \cite{ricsol2} for a survey. We note that some definitions of Ricci solitons, such as in \cite{ricsol}, assume completeness of the metric.

The  
$c_1=0$, $c_2=0$ case is the equation for homotheties of the metric, i.e.\ the solutions give rise to homothetic vector fields $X^a$ (conformal Killing with constant divergence $\la$).

The situation with $c_1=c_2=\la=0$ is Killing's equation, the equation determining infinitesimal isometries of the metric.

The $c_1=1$, $c_2=-\frac{1}{n-2}$ cases are special cases of the Einstein-Weyl equation in conformal geometry for $n>2$; see \cite{EW} for definitions.

The $c_1=1$, $c_2=-\frac{1}{n-1}$, $\la=0$ cases are the equations determining whether a metric projective structure admits a skew-symmetric Ricci tensor representative in its projective class. See \cite{thesis}, \cite{skewricci2D} for further details. We remark that coefficient $\frac{1}{n-1}$ appears in the equation because the projective Schouten tensor $\P$ is related to the Ricci tensor of a metric connection by $\P_{ab}=\frac{1}{n-1}R_{ab}$.

The $c_1=1$, $c_2=\frac{1}{2}$ case is the vacuum near-horizon geometry equation. This equation is recently studied in general relativity in the context of existence of extremal black holes. In this equation $\la$ is the cosmological constant of the spacetime. 
See \cite{static}, \cite{Jezierski}, \cite{KL3} for further discussions and refer to Table \ref{GRSexamples}.

 \begin{table}[h]
 \begin{tabular}{|c|c|c|c|} \hline
 Equation & $c_1$ & $c_2$ & $\la$ \\ \hline\hline
 Killing's equation &  $0$ & $0$ & $0$ \\ \hline
 Equation for homotheties & $0$ & $0$ & $\ast$ \\ \hline
 Ricci solitons & $0$ & $-1$ & $\ast$ \\ \hline
 Cases of Einstein-Weyl & $1$ & $-\frac{1}{n-2}$ & $\ast$ \\ \hline
\begin{tabular}{c} Metric projective structures with \\ skew-symmetric Ricci tensor in projective class\end{tabular} & $1$ & $-\frac{1}{n-1}$ & 0  \\ \hline
  Vacuum near-horzion geometry equation & $1$ & $\frac{1}{2}$ & $\ast$ \\ \hline
  \end{tabular}

 \caption{Examples of generalised Ricci solitons}
 \label{GRSexamples}
 \end{table}

The generalised Ricci soliton equations constitute an overdetermined system of PDEs of finite type. They are linear in $X_a$ in the case where $c_1=c_2=0$ (homotheties), non-homogeneous linear in $X_a$ when $c_1=0$ (Ricci solitons), and quadratic in $X_a$ in the case $c_1=1$. The leading term in all these equations have the same symbol as the differential operator $X_a \mapsto \nd_{(a}X_{b)}$.

The first part of the paper (Sections 2-5) is devoted to prolongation of (\ref{grs0}) to derive algebraic constraints and obstructions in dimensions $2$ and $3$ (in some cases),  while the second part (Sections 6-8) is devoted to constructing explicit examples mainly in dimension $2$. In particular, in the second part of the paper we put in a broader context such important examples as the celebrated Hamilton's cigar Ricci soliton (see Proposition \ref{cigar}), extremal Kerr black hole horizon (see Corollary \ref{exkerr}), and reduction of dKP equation to 2 dimensions in Lorentzian signature  (see Proposition \ref{reduceddKP}). We also give some 3 dimensional examples in Example \ref{3dricsol} and in Remark \ref{3dew}.

We raise and lower indices of tensor fields with respect to the metric $g$, and we will not distinguish between $1$-forms and vector fields when convenient.

\section{Generalised Ricci solitons: Prolongation and closed system}

We now prolong equation (\ref{grs0}) to get a closed system. 

If we denote $F_{ab}=\nd_{[a}X_{b]}$, then we get
\begin{equation}\label{gs2}
\nd_{a}X_{b}+c_1 X_aX_b - c_2 R_{ab} =F_{ab}+ \la g_{ab}.
\end{equation}
We shall prolong (\ref{gs2}) and get differential constraints for $g$ to admit a solution $X$ to (\ref{gs2}). We will concentrate on dimensions 2 and 3 (with $c_1=0$ in the 3 dimensional case), although some parts of the prolongation will be valid in any higher dimensions.

In any dimension, the prolongation gives
\begin{equation}\boxed{\begin{aligned}
\nd_{a}X_{b}=&-c_1 X_aX_b +c_2 R_{ab}+F_{ab}+ \la g_{ab}\\
\nd_aF_{bc}=&c_2(\nd_{b}R_{ca}-\nd_c R_{ba})+R_{bc}{}^d{}_aX_d+2c_1 F_{cb}X_a+c_1 X_bF_{ca}-c_1 X_c F_{ba}\\
&+\la c_1 X_b g_{ca}- \la c_1 X_c g_{ba}+c_1c_2 X_b R_{ca}-c_1 c_2 X_c R_{ba}.
\end{aligned}}\label{gspr}\end{equation}
Note that $X_a$ appears quadratically in the closed system. Note also, that both equations of the closed system have inhomogeneous terms; these, such as for example $c_2\nabla_aR_{bc}$, depend on the parameters $c_1,c_2,\lambda$ and geometric quantities associated with the metric, but do not depend on the unknowns $X_a$ and $F_{ab}$.  

With redefined constants, for $c_1=1$, we get
\begin{align*}
\nd_{a}X_{b}=&-X_aX_b +\tilde c_2 R_{ab}+F_{ab}+ \tilde \la g_{ab}\\
\nd_aF_{bc}=&\tilde c_2(\nd_{b}R_{ca}-\nd_c R_{ba})+R_{bc}{}^d{}_aX_d+2 F_{cb}X_a+ X_bF_{ca}-X_c F_{ba}\\
&+\tilde \la X_b g_{ca}- \tilde \la X_c g_{ba}+\tilde c_2 X_b R_{ca}-\tilde c_2 X_c R_{ba},
\end{align*}
while for $c_1=0$, we obtain
\begin{align*}
\nd_{a}X_{b}=&c_2 R_{ab}+F_{ab}+ \la g_{ab}\\
\nd_aF_{bc}=&c_2(\nd_{b}R_{ca}-\nd_c R_{ba})+R_{bc}{}^d{}_aX_d.
\end{align*}

\section{2D generalised Ricci solitons: Theory}\label{2DgRsT}

\subsection{Prolongation and constraints of 2D generalised Ricci soliton equations}
In 2 dimensions, we have 
$R_{abcd}=K(g_{ac}g_{bd}-g_{ad}g_{bc})$, $R_{ab}=\frac{R}{2}g_{ab}=K g_{ab}$, where $R$ is the scalar curvature and $K=\frac{R}{2}$ is the Gauss curvature. We can also use the volume form $\epsilon_{ab}$ to dualise all 2-forms. In the Riemannian signature case we have $\epsilon^{ab}\epsilon_{ac}=\de^b{}_c$, $\epsilon^{ab}\epsilon_{ab}=2$ while in the Lorentzian signature case we have $\epsilon^{ab}\epsilon_{ac}=-\de^b{}_c$ and $\epsilon^{ab}\epsilon_{ab}=-2$, so we take $\epsilon^{ab}\epsilon_{ac}=e \de^b{}_c$, $\epsilon^{ab}\epsilon_{ab}=2 e$ where $e=\{\pm 1\}$ depending on the signature of the metric. We therefore can write $F_{ab}=\frac{e}{2}\epsilon_{ab}F$ where $F=\epsilon^{ab}F_{ab}$. Also note that we have $\epsilon_{ab}\epsilon_{cd}=e(g_{ac}g_{bd}-g_{ad}g_{bc})$.

The second equation of (\ref{gspr}) in 2 dimensions reduces to:
\begin{align*}
\frac{e}{2}\epsilon_{bc}\nd_aF=&c_2(g_{ca}\nd_{b}K-g_{ab}\nd_c K)+K X_b g_{ca}-K X_c g_{ba}-\frac{3e}{2}c_1 \epsilon_{bc} F X_a\nonumber\\
&+\la c_1 X_b g_{ca}- \la c_1 X_c g_{ba}+c_1c_2 K X_b g_{ca}-c_1 c_2 K X_c g_{ba},
\end{align*}
so that contracting throughout with $\epsilon^{bc}$ gives
\begin{align*}
\nd_aF=&2c_2 \epsilon^b{}_a\nd_{b}K+2(\la c_1+K(1+c_1c_2)) X_b\epsilon^{b}{}_a-3c_1 F X_a.
\end{align*}
The prolonged system of the generalised soliton equation (\ref{grs0}) therefore reduces to  
\begin{equation}\label{gssys}
\boxed{\begin{aligned}\nd_{a}X_{b}=&-c_1 X_aX_b +c_2 K g_{ab}+\frac{e}{2}F \epsilon_{ab}+ \la g_{ab}\\
\nd_aF=&-3 c_1 X_a F+2 \epsilon^b{}_a(c_2 \nd_bK+(1+c_1c_2)X_b K+\la c_1 X_b)\end{aligned}}
\end{equation}
in $2$ dimensions.

Let us call $$L_b=c_2 \nd_bK+(1+c_1c_2)X_b K+\la c_1 X_b=c_2 \nd_b K+((1+c_1c_2)K+\la c_1)X_b.$$ Then from differentiating 
\begin{align*}
\nd_aF&=-3 c_1 X_a F+2 \epsilon^b{}_aL_b
\end{align*}
we obtain
\begin{align*}
\nd_b\nd_aF&=-3 c_1 (\nd_bX_a) F-3 c_1 X_a (\nd_b F)+2 \epsilon^c{}_a(\nd_b L_c),
\end{align*}
and upon skewing with the volume form $\epsilon^{ba}$ (recall that $\epsilon^{ba}\nd_{b}X_a=\epsilon^{ba}F_{ba}= F$) we get
\begin{align}\label{const1}
0&=-3 c_1 F^2+6e c_1 X_bL^b+2e\nd_b L^b.
\end{align}
Now a further computation yields that
\begin{align*}
\nd_eL_a=&c_2 \nd_e\nd_aK+(1+c_1c_2)(\nd_eX_a)K+\la c_1 (\nd_e X_a)+(1+c_1c_2)X_a\nd_eK,
\end{align*}
from which we obtain by tracing indices
\[
\nd_aL^a=c_2 \De K+(1+c_1c_2)(\nd_aX^a)K+\la c_1 (\nd_a X^a)+(1+c_1c_2)X^a\nd_aK.
\]

Since
\[
\nd_aX^a=-c_1 X_aX^a+2 c_2 K+2 \la, 
\]
this gives
\begin{align*}
\nd_aL^a=&c_2 \De K+(1+c_1c_2)(-c_1 X_aX^a+2 c_2 K+2 \la)K\\
&+\la c_1 (-c_1 X_aX^a+2 c_2 K+2 \la)+(1+c_1c_2)X^a\nd_aK,
\end{align*}
and so (\ref{const1}) is given by
\begin{align*}
0=&-3 c_1 F^2+6 e c_1 (c_2 X^b\nd_bK+(1+c_1c_2)X^bX_b K+\la c_1 X^bX_b)\\
&+2e\bigg(c_2 \De K+(1+c_1c_2)(-c_1 X_aX^a+2 c_2 K+2 \la)K\\
&+\la c_1 (-c_1 X_aX^a+2 c_2 K+2 \la)+(1+c_1c_2)X^a\nd_aK\bigg).
\end{align*}

Collecting like terms together we obtain the first differential constraint as
\begin{equation}\boxed{\begin{aligned}
-3 c_1 &F^2+4 e c_1 \big((1+c_1c_2)K+\la c_1\big) X^bX_b +2e(1+4c_1c_2)X^a\nd_aK\\
&+2e\bigg(c_2 \De K+\big((1+c_1c_2)K+\la c_1\big)(2 c_2 K+2 \la)\bigg)=0.
\end{aligned}}\label{gscons}\end{equation}

We can differentiate (\ref{gscons}) further and use the closed system to obtain a second constraint, from which we try to solve for $X_a$. It turns out that to derive this second constraint is tough and technically demanding, and so it might be more worthwhile to look at certain special cases instead. 
 
\begin{remark}
From the first constraint we see that in order for the quadratic term $X^aX_a$ to vanish, we need either $c_1=0$ or $K$ is of constant curvature with $K=-\frac{\la c_1}{1+c_1c_2}$. In the second situation, this implies $\nd_a K=0$ (and hence $\De K=0$), and the differential constraint reduces to 
\[
0=-3 c_1 F^2, 
\]
so that $F=0$ if $c_1 \neq 0$. We also see that the linear term involving $X_a\nd^a K$ vanishes at the critical value of $1+4 c_1c_2=0$. 
\end{remark}

\begin{remark}\label{critical}
When $1+4c_1c_2=0$, the term involving $X^a\nd_aK$ in (\ref{gscons}) vanishes and the equation reduces to  
\begin{align*}
&2e\bigg(c_2 \De K+((1+c_1c_2)K+\la c_1)(2 c_2 K+2 \la)\bigg)\\
&-3 c_1 F^2+4e c_1 ((1+c_1c_2)K+\la c_1) X^bX_b=0.
\end{align*}
Setting $c_1=1$ and $c_2=-\frac{1}{4}$, we obtain
\begin{align*}
-\frac{e}{2} \De K+\frac{e}{4}(3 K+4\la)(4 \la-K)-3 F^2+ e(3 K+4\la ) X^bX_b=0.
\end{align*}
Further setting $\la=0$, and taking $e=1$ in the Riemannian setting gives the differential constraint
\begin{align}\label{specialc2}
\De K+\frac{3}{2}K^2+6 F^2- 6 K X^bX_b=0.
\end{align}
For an explicit example in Lorentzian signature see the end of Section \ref{nullc11}.
\end{remark}

\subsection{2D gradient generalised Ricci soliton}\label{ggrs2d}
We call a solution $(g, X)$ to (\ref{grs0}) a gradient generalised soliton if $F_{ab}=\nd_{[a}X_{b]}=0$. In such case $X_a=\nd_a f$ for some function $f$ locally.

In 2 dimensions, in the case when $F=0$, the second equation in the prolonged system (\ref{gssys}) forces $L_a$ to vanish, i.e. 
\begin{align*}
\boxed{c_2 \nd_b K+((1+c_1c_2)K+\la c_1)X_b=0.}
\end{align*}
In this case 
$X_a$ is necessarily given by
\begin{align}\label{ggRsal}
X_a=-\frac{c_2 \nd_a K}{(1+c_1c_2)K+\la c_1}.
\end{align}
We therefore have
\begin{proposition}\label{gradientprop}
Let $(M,g)$ be a (pseudo)-Riemannian 2-manifold, with the Gauss curvature $K$ of $g$ not equal to $-\frac{\la c_1}{1+c_1 c_2}$. Then in order for $M$ to admit a solution to the gradient generalised Ricci soliton equations, we must necessarily have (\ref{grs0}) satisfied for $X_a$ given by (\ref{ggRsal}). Conversely, suppose that equation (\ref{grs0}) is satisfied for some $X_a$ given by (\ref{ggRsal}), then $M$ admits a solution to the gradient generalised Ricci soliton equations.  
\end{proposition}
As a Corollary to Proposition \ref{gradientprop}, we have:
\begin{corollary}
The local obstruction for a 2-dimensional (pseudo)-Riemannian metric $g$ with $K\neq-\frac{\la c_1}{1+c_1 c_2}$ to admit a gradient generalised Ricci soliton $(g,X)$ is the obstruction tensor $\Theta_{ab}$ given by
\end{corollary}
\begin{align*}
\Theta_{ab}=-\frac{c_2 \nd_a\nd_b K}{(1+c_1c_2)K+\la c_1}+\frac{c_2(1+2 c_1c_2)\nd_a K \nd_b K}{((1+c_1 c_2)K+\la c_1)^2}-(c_2 K+\la) g_{ab}. 
\end{align*}
This vanishes if and only if the pair 
\[
(g,X)=\left(g,-\frac{c_2 \nd_aK}{(1+c_1c_2)K+\la c_1}\right)
\]
is a gradient generalised Ricci soliton.
\begin{example}
The metric on $\R^2$ given by
\begin{align*}
g=\frac{y^4+6}{6}\der x^2+\frac{6}{y^4+6}\der y^2
\end{align*}
cannot lead to a gradient generalised Ricci soliton $(g,X)$ for generic chosen values of $c_1=1$, $c_2=1$, $\la=1$. For this example, $K=-y^2$, and our formula for $X$ yields
\[
X=-\frac{2y}{2y^2-1}\der y,
\]
so that $\der X=0$
and plugging this formula for $X$ back into the generalised Ricci soliton equation gives
\begin{align*}
\Theta_{ab}\der x^a \der x^b=\frac{4y^8-9y^6+27y^4-54y^2+18}{18(2y^2-1)}\der x^2+\frac{6(4y^6-7y^4+13y^2+1)}{(y^4+6)(2y^2-1)^2}\der y^2
\end{align*} 
which is non-zero. 
\end{example}

\subsection{2D Ricci Solitons}{\bf The case where $c_1=0$, $c_2=-1$.} Note that if $g$ is a metric of constant curvature, the generalised Ricci soliton equation with $c_1=0$, $c_2=-1$ is just an equation for homotheties. We therefore exclude the pairs $(g,X)$ such that $g$ is a metric of constant curvature from the analysis in this section. Since these solutions belong to the category of homotheties, we will discuss them in Section \ref{homotheties}.

For the convenience of the reader we present the closed system (\ref{gssys}) for the generalised Ricci solitons and its integrability conditions (\ref{gscons}) specialised to the case $c_1=0$ and $c_2=-1$. This gives the equations describing the proper Ricci solitons:
\begin{equation}\label{ricsoleq}
\boxed{\begin{aligned}\nd_{a}X_{b}=&- K g_{ab}+\frac{e}{2}F \epsilon_{ab}+ \la g_{ab}\\
\nd_aF=&2 \epsilon^b{}_a(-\nd_bK+X_b K)\end{aligned}}
\end{equation}
\begin{equation}\label{ricsolconst}\boxed{\begin{aligned}
X^a\nd_aK- \De K+2 K \la-2 K^2=0.
\end{aligned}}\end{equation}    
Excluding the constant curvature case, the proper Ricci solitons can be characterised by the following Theorem.
\begin{theorem}\label{ricsolthm}
A 1-form $X$ defining a 2-dimensional Ricci soliton $(g,X)$ is of the form
\begin{align*}
X_a=\frac{1}{\rho}\left(-2(\nd^c K)\nd_c \De K+2(3\la-5 K)M\right)\epsilon_{ab}\nd^b K+\frac{1}{\rho}\left(\De K+2K^2-2 K\la\right)\epsilon_{ab}\nd^b M,
\end{align*}
where $K$ is the Gauss curvature of $g$, 
\[M=g^{ab}\nd_aK \nd_bK\quad{and}\quad\rho=\epsilon^{ab}\nd_aK\nd_bM,\]
provided that $\rho \neq 0$.

If $\rho=0$, then $X$ is of the form
\begin{align*}
X_c=-\frac{1}{\nu}\left(\epsilon^{ab}\nd_a \De K\nd_b K-\frac{1}{2}MF\right)\epsilon_{cd}\nd^dK+\frac{1}{\nu}\big(\De K+2 K^2-2 K \la\big)\epsilon_{cd}N^d,
\end{align*}
where 
\[N_c=(\nd_a\nd_cK)\epsilon^{ab}\nd_bK\quad and\quad\nu=\epsilon^{ab}\nd_aKN_b,\] provided that $\nu \neq 0$.

If $\nu=0$, in the Riemannian case, $e=1$, the vector $X$ is locally a gradient and this situation is a specialisation of the results of Proposition \ref{gradientprop} with $c_1=0$ and $c_2=-1$.
\end{theorem}

\begin{proof}
For the $c_1=0$ equations, we have
\begin{align*}
0=&c_2 (\De K+2K^2)+2 K\la +X^a\nd_aK,
\end{align*}
and we can set $c_2=-1$ to get
\begin{align}\label{ricconst1}
X^a\nd_aK=(\De K+2K^2)-2 K\la.
\end{align}
Differentiating this gives
\begin{align}\label{paw}
0=&-\nd_a \De K-4 K \nd_a K+2 (\nd_a K) \la +(\nd_a X^b)\nd_bK+X^b \nd_a \nd_b K\nonumber\\
=&-\nd_a \De K-4 K \nd_a K+2 (\nd_a K) \la +(\frac{e}{2}\epsilon_a{}^bF+\de_a{}^b (\la- K))\nd_bK+X^b \nd_a \nd_b K,
\end{align}
upon which contracting by $\nd^a K$ gives
\begin{align*}
0=&-(\nd^a K)\nd_a \De K-4 K \nd^a K\nd_a K+2 \nd^a K(\nd_a K) \la\\ &+((\la- K))\nd^b K \nd_bK+X^b \nd^a K \nd_a \nd_b K\\
=&- (\nd^a K)\nd_a \De K+(3\la-5 K)\nd^b K \nd_bK+X^b \nd^a K \nd_a \nd_b K.
\end{align*}
Let us call $(\nd_a K)(\nd^a K)=M$, so that
$\nd_a M=2(\nd_b K) (\nd_a \nd^b K)$.
Then we have the second constraint given by
\begin{align}\label{ricconst2}
X_a \nd^a M=2(\nd^a K)\nd_a \De K-2(3\la-5 K)M.
\end{align}
Now assuming $\rho:=\epsilon^{ab}\nd_a K \nd_b M \neq 0$, the vectors $\epsilon^a{}_{b}\nd^b K$ and $\epsilon^a{}_{b}\nd^b M$ form a basis and
the equations (\ref{ricconst1}) and (\ref{ricconst2}) give the components of $X_a$ in this basis. We get: 
\begin{align*}
X_a=\frac{1}{\rho}\left(-2(\nd^c K)\nd_c \De K+2(3\la-5 K)M\right)\epsilon_{ab}\nd^b K+\frac{1}{\rho}\left(\De K+2K^2-2 K\la\right)\epsilon_{ab}\nd^b M,
\end{align*}
so that $X_a$ is completely determined by invariants of the metric. Then, plugging $X_a$ from this formula back into equation (\ref{grs0}) with $c_1=0$, $c_2=-1$ gives us local if-and-only-if obstructions for the metric to admit any Ricci soliton $(g,X)$.

There is an alternative formula for $X$, again determined by invariants of the metric, which is convenient for us to use when $\rho=0$. To get it we first need a formula
\begin{align}\label{paw1}
0=&-\epsilon^{ab}\nd_a \De K\nd_b K+\frac{1}{2}MF+X^c \epsilon^{ab}(\nd_a \nd_c K)\nd_b K,
\end{align}
which is obtained by contracting the differential constraint (\ref{paw})
with $\epsilon^{ab}\nd_bK$. Then we define 
\[
N_c=(\nd_a\nd_cK)\epsilon^{ab}\nd_bK,
\]
and obtain its projection 
\begin{align}\label{pav2}
X^cN_c=&\epsilon^{ab}\nd_a \De K\nd_b K-\frac{1}{2}MF
\end{align}
onto $X$,  
by (\ref{paw1}). We also have 
\begin{align*}
X^c\nd_cK=&\De K+2 K^2-2 K \la
\end{align*}
by (\ref{ricconst1}).

Using these projections and assuming
\begin{align*}
\nu:=&\epsilon^{ab}\nd_aKN_b\neq 0,
\end{align*}
we express $X_c$ in the basis given by $\epsilon_{cd}N^d$ and $\epsilon_{cd}\nd^dK$ obtaining:
\begin{align*}
X_c=-\frac{1}{\nu}\left(\epsilon^{ab}\nd_a \De K\nd_b K-\frac{1}{2}MF\right)\epsilon_{cd}\nd^dK+\frac{1}{\nu}\big(\De K+2 K^2-2 K \la\big)\epsilon_{cd}N^d.
\end{align*}
In the situation where $\rho=0$, we can use the above formula to compute $X$ provided $\nu \neq 0$. 

In the case where $\rho=0$ and $\nu=0$ we can write 
\[
\nd_aM=\ell \nd_a K 
\]
for some function $\ell$.
Also note that in such case $\nu$ can be expressed in terms of $\ell$, $\De K$ and $M$ as follows:
\begin{align*}
\nu=&\epsilon^{ab}\nd_aK(\nd_b\nd_eK)(\epsilon^{ed}\nd_dK)\\
=&e(g^{ae}g^{bd}-g^{ad}g^{be})\nd_aK\nd_dK(\nd_b\nd_eK)\\
=&e((\nd^aK)(\nd^bK)\nd_a\nd_bK-M\De K)\\
=&e(\frac{1}{2}\nd^aK\nd_aM-M\De K)\\
=&e\left(\frac{\ell }{2}-\De K\right)M.
\end{align*}

To complete the proof we have to prove that in the Riemannian $e=1$ case, with $\rho=0$ and $\nu=0$, we have $F=0$. 

For this we will use formula (\ref{pav2}), and first show that  
\begin{align}\label{pav3}
\epsilon^{ab}\nd_aK \nd_b \De K=0,
\end{align}
and then show that $X^cN_c=0$.

Indeed, since $\rho=0$ we have $\nd_aM=\ell \nd_aK$, so by differentiating and anti-symmetrising, we get that 
\begin{align}\label{pav4}
\epsilon^{ab}\nd_aK \nd_b\ell=0. 
\end{align}
On the other hand, the assumption about the Riemannian signature ($e=1$), implies
\[
M=\nd_aK \nd^aK >0,
\] 
because we excluded the constant curvature case. Hence the condition $\nu=0$
gives $\ell=2 \De K$. This, when compared with (\ref{pav4}), gives (\ref{pav3}), as claimed. 

Moreover $\nu=0$ gives:
\[
N_c=j \nd_cK 
\]   
for some function $j$,
which implies
\[
0=\frac{1}{2}\epsilon^{ab}\nd_aM \nd_b K=\nd^cK(\nd_a\nd_cK)\epsilon^{ab}\nd_bK=j\nd_cK\nd^cK=jM. 
\]
Therefore, the fact that $M>0$ is non-zero, implies $j=0$, so $N_a=0$. Now, (\ref{pav2}) gives $F=0$, again by the Riemannian condition $M>0$. Hence $X_a$ is a gradient. 
\end{proof}

\begin{corollary}
The local obstruction for a 2-dimensional (pseudo)-Riemannian metric g with $\rho \neq 0$ to admit a Ricci soliton $(g, X)$ is the obstruction tensor $\Theta^{(1)}_{ab}$ given by
\[
\Theta^{(1)}_{ab}=\nd_{(a} X^{(1)}_{b)}+K g_{ab}-\la g_{ab}
\] 
where 
\begin{align*}
X^{(1)}_a=\frac{1}{\rho}\left(-2(\nd^c K)\nd_c \De K+2(3\la-5 K)M\right)\epsilon_{ab}\nd^b K+\frac{1}{\rho}\left(\De K+2K^2-2 K\la\right)\epsilon_{ab}\nd^b M.
\end{align*}
The obstruction tensor $\Theta^{(1)}_{ab}$ vanishes if and only if
\[
(g,X)=(g,X^{(1)})
\]
is a Ricci soliton. 
\end{corollary}

\begin{remark}\label{fremark}
When $\rho=0$ and $\nu \neq0$, the formula for $X_a$ obtained in the proof of Theorem \ref{ricsolthm} given by
\begin{align}\label{xf1}
X_c=-\frac{1}{\nu}\left(\epsilon^{ab}\nd_a \De K\nd_b K-\frac{1}{2}MF\right)\epsilon_{cd}\nd^dK+\frac{1}{\nu}\big(\De K+2 K^2-2 K \la\big)\epsilon_{cd}N^d
\end{align}
still involves the unknown quantity $F$. To solve for $F$, we can substitute $X_c$ given by (\ref{xf1}) back
into the first three Ricci soliton equations (\ref{ricsoleq}) and use the last two equations (\ref{ricsoleq}) to eliminate derivatives of $F$ that appear. The first three equations will then only involve differential invariants of the metric and $F$. The function $F$ can thus be algebraically determined in terms of the metric invariants, by solving one of these equations. Now the substitution of this $F$ back into (\ref{xf1}) determines $X_c$ completely. To obtain obstructions in this case one inserts $X_c$ back into the two Ricci soliton equations (out of the first three equations (\ref{ricsoleq})) which were not used in determining $F$. We therefore will have at most $2$ scalar local obstructions for $g$ with $\rho=0$, $\nu \neq0$ to admit a Ricci soliton. We also remark that another procedure for finding obstructions may be more convenient to use. This consists in inserting $X$, as in (\ref{xf1}), into the last two equations (\ref{ricsoleq}), and then in using the integrability conditions $(\nabla_a\nabla_b-\nabla_b\nabla_a)F=0$ by applying the covariant derivatives on $\nabla_aF$. Since the derivatives $\nabla_aF$, after an insertion of $X$ in them, are only 
expressible in terms of the metric invariants and $F$, this leads either to a formula for $F$, or to an obstruction independent of $F$, expressed in terms of the metric invariants only. 

Explicitly, if we write $X_c=A_c+B_c F$, where
\begin{align*}
A_c=&-\frac{1}{\nu}\left(\epsilon^{ab}\nd_a \De K\nd_b K\right)\epsilon_{cd}\nd^dK+\frac{1}{\nu}\big(\De K+2 K^2-2 K \la\big)\epsilon_{cd}N^d,\\
B_c=&\frac{1}{2 \nu}M \epsilon_{cd}\nd^d K, 
\end{align*}
then the last 2 equations of (\ref{ricsoleq}) give
\begin{align*}
\nd_aF=C_a+FE_a,  
\end{align*}
where
\begin{align*}
C_a=&2 \epsilon_{ab}\nd^b K-2\epsilon_{ab}A^b K,\\
E_a=&-2 \epsilon_{ab}B^b K. 
\end{align*}
The integrability condition $(\nabla_a\nabla_b-\nabla_b\nabla_a)F=0$ then implies that
\begin{equation}\label{linearf}
0=\epsilon^{ab}\nd_aC_b+\epsilon^{ab}C_aE_b+(\epsilon^{ab}\nd_aE_b) F. 
\end{equation}
If $\epsilon^{ab}\nd_aE_b=0$, then $\epsilon^{ab}\nd_aC_b+\epsilon^{ab}C_aE_b$
is an obstruction. Otherwise, we can solve for $F$ in (\ref{linearf}) and plug this formula back into (\ref{xf1}) to determine $X$ completely in terms of metric invariants. 
We explicitly show this alternative procedure in the following example, where we have $\epsilon^{ab}\nd_aE_b=0$. The obstruction $\epsilon^{ab}\nd_aC_b+\epsilon^{ab}C_aE_b$ is then a $5^{\rm th}$ order ODE on a single function $f(x)$. 
\end{remark}
\begin{example}\label{soi}

Consider the metric
$$g={\rm e}^{2f}(\der x^2 +\der y^2),$$
with a real function $f$ of variable $x$ only, $f=f(x)$. 
Our aim in this example is to find obstructions for $g$ to admit a Ricci soliton. 

We easily calculate that:
$$ \rho\equiv 0,\quad\quad\quad \nu=-{\rm e}^{-10f}f'(2f'f''-f^{(3)})^3.$$
Thus, assuming that $f'(2f'f''-f^{(3)})\neq 0$, because $\rho\equiv 0$, we have to calculate $X_c$ using formula (\ref{xf1}). This, according to Theorem \ref{ricsolthm}, is the necessary form of $X_c$, for it to be a Ricci soliton $(g,X)$. Explicitly, (\ref{xf1}) gives:
\begin{equation} X=\frac{f^{(4)}-4f'f^{(3)}-4{f''}^2+4{f'}^2f''-2{\rm e}^{2f}\lambda f''}{f^{(3)}-2f'f''}~{\rm e}^{-2f}~\partial_x~ +~\frac{F}{2f'}~\partial_y.\label{roz}\end{equation}
Here $F$ is the unknown function $F=F(x,y)$ responsible for the skew symmetric part of $\nabla_a X_b$. We now insert this $X$ to the last two of the closed system equations (\ref{ricsoleq}) and solve for the derivatives $F_x$ and $F_y$. We get:
$$F_x=\frac{f''}{f'}F,\quad\quad F_y=\frac{2{\rm e}^{-2f}(2{\rm e}^{2f}\lambda {f''}^2+4{f''}^3+{f^{(3)}}^2-f''f^{(4)})}{f^{(3)}-2f'f''}.$$
We now have to assure that $F_{xy}=F_{yx}$, which requires that $f=f(x)$ satisfies a certain $5^{\rm th}$ order ODE. This is precisely the obstruction obtained in (\ref{linearf}). From this we have the $5^{\rm th}$ derivative $f^{(5)}$. Now, we insert (\ref{roz}) in the first four closed system equations (\ref{ricsoleq}). By using the computed $F_x$, $F_y$, and their integrability condition, which gives us the $5^{\rm th}$ derivative of $f$, we see that these four equations reduce to a single one, which is 
\begin{equation}
f^{(4)}(f''-{f'}^2)-4{f'}^4f''-2{\rm e}^{2f}\lambda {f''}^2+2{f'}^2{f''}^2-4{f''}^3+({\rm e}^{2f}\lambda f'+4{f'}^3+f'f'')f^{(3)}-{f^{(3)}}^2=0.\label{czw}\end{equation} 
This gives a lower order obstruction for $g$ to admit any Ricci soliton. 

We now solve for $f^{(4)}$ from this equation, and recalculate our $X$, $F_x$ and $F_y$, obtaining:
$$ X=\frac{f^{(3)}-3f'f''-{\rm e}^{2f}\lambda f'}{f''-{f'}^2}~{\rm e}^{-2f}~\partial_x~ +~\frac{F}{2f'}~\partial_y$$
and 
$$F_x=\frac{f''}{f'}F,\quad\quad F_y=\frac{2{\rm e}^{-2f}f'({\rm e}^{2f}\lambda {f''}+2{f'}^2f''+{f''}^2-f'f^{(3)})}{f''-{f'}^2}.$$
Having this we observe, that now $F_{xy}=F_{yx}$ is equivalent to the equation (\ref{czw}). Hence we can forget about the $5^{\rm th}$ order ODE for $f$, we have used previously. We have just proved that this is implied by (\ref{czw}). 

So the conclusion, up to now, is that (\ref{czw}) is the only condition needed for $g$ to admit any Ricci soliton. 

With this equation satisfied, we can solve for $F$. Integration of $F_x$ gives:
$$F=hf',$$
with the function $h$ depending on variable $y$ only, $h=h(y)$. Now, insertion of this into the formula for $F_y$ gives the following equation:
$$h'=\frac{2{\rm e}^{-2f}({\rm e}^{2f}\lambda {f''}+2{f'}^2f''+{f''}^2-f'f^{(3)})}{f''-{f'}^2}.$$
Since the left hand side of this equation depends only on $y$, and the right hand side only on $x$, then both sides must be equal to a real constant, say $2a$. 
Then we have $h=2ay +2b$, with a real constant $b$; the function $f$ must staisfy the third order ODE:
\begin{equation}f^{(3)}f'-{f''}^2+(a{\rm e}^{2f}-2{f'}^2-{\rm e}^{2f}\lambda)f''-a{\rm e}^{2f}{f'}^2=0.\label{trz}\end{equation}
It follows that, if $f$ satisfies this equation, then it automatically satisfies equation (\ref{czw}). In other words, this equation is the first integral for (\ref{czw}). 
\end{example}

In such a way, we solved for $X$, $F$, and the only equation to be satisfied for $g$ to have a Ricci soliton, is just the third order ODE (\ref{trz}). We summarise the consideration in this example in the following proposition.

\begin{proposition}
The metric
$$\boxed{g={\rm e}^{2f}(\der x^2 +\der y^2),}$$
admits a Ricci soliton if and only if the function $f=f(x)$ satisfies a third order ODE:
$$\boxed{f^{(3)}f'-{f''}^2+(a{\rm e}^{2f}-2{f'}^2-{\rm e}^{2f}\lambda)f''-a{\rm e}^{2f}{f'}^2=0.}$$
If this equation is satisfied the soliton is given by a vector field 
$$\boxed{X~=~\frac{{\rm e}^{-2f}f''+\lambda-a }{f'}~\partial_x+(ay+b)~\partial_y.}$$
Here $a$ and $b$ are real constants. The soliton is a gradient Ricci soliton if and only if $a=b=0$. 

\end{proposition}

\subsection{2D homotheties}{\bf The case where $c_1=0$, $c_2=0$.}\label{homotheties}

Again for the convenience of the reader we present the closed system (\ref{gssys}) for the homothety equations and its integrability conditions (\ref{gscons}) by setting $c_1=c_2=0$:
\begin{equation}\label{homeq}
\boxed{\begin{aligned}\nd_{a}X_{b}=&\frac{e}{2}F \epsilon_{ab}+ \la g_{ab}\\
\nd_aF=&2 \epsilon^b{}_aX_b K\end{aligned}}
\end{equation}
\begin{equation}\label{homconst}\boxed{\begin{aligned}
X^a\nd_aK+2 \la K=0.
\end{aligned}}\end{equation}    

Our characterisation of homotheties in 2 dimensions is given by the following theorem:
\begin{theorem}\label{homothetythm}
A 1-form $X$ defining a 2-dimensional homothety pair $(g,X)$ is of the form
\begin{align*}
X_a=\frac{1}{\rho}\left(6 M \la \epsilon_{ab}\nd^b K-2 K\la\epsilon_{ab}\nd^b M\right),
\end{align*}
provided that $\rho \neq 0$.

If $\rho=0$, then $X$ is of the form
\begin{align*}
X_c=\frac{ MF}{2 \nu}\epsilon_{cd}\nd^dK-\frac{2 K \la}{\nu}\epsilon_{cd}N^d,
\end{align*}
(where again $N_c=(\nd_a\nd_cK)\epsilon^{ab}\nd_bK$ and $\nu=\epsilon^{ab}\nd_aKN_b$), provided that $\nu \neq 0$.

If $\nu=0$, in the Riemannian case, $e=1$, the vector $X$ is locally a gradient and we return to the situation of Proposition \ref{gradientprop} as before. 
\end{theorem}
\begin{proof}
In the case where $c_1=c_2=0$ we obtain from (\ref{gscons}) that
\begin{align}\label{hom1}
 X^a\nd_aK=&-2 K\la.
\end{align}
Differentiating this equation gives
\begin{align}\label{pawhom}
0=&2 (\nd_a K) \la +(\nd_a X^b)\nd_bK+X^b \nd_a \nd_b K\nonumber\\
=&2 (\nd_a K) \la +(\frac{e}{2}\epsilon_a{}^bF+\de_a{}^b \la)\nd_bK+X^b \nd_a \nd_b K,
\end{align}
and contracting by $\nd^a K$ results in
\begin{align*}
0=&3\la M +X^b \nd^a K \nd_a \nd_b K\\
=&3\la M+\frac{X^a\nd_a M}{2}.
\end{align*}
Thus, the second constraint is given by
\begin{align}\label{hom2}
X_a \nd^a M=&-6 M \la.
\end{align}
Now assuming $\rho:=\epsilon^{ab}\nd_a K \nd_b M \neq 0$, the vectors $\epsilon^a{}_{b}\nd^b K$ and $\epsilon^a{}_{b}\nd^b M$ form a basis. The equations
(\ref{hom1}) and (\ref{hom2}) give the components of $X_a$ in this basis. Therefore we get
\begin{align*}
X_a=\frac{1}{\rho}\left(6 M \la \epsilon_{ab}\nd^b K-2 K\la\epsilon_{ab}\nd^b M\right).
\end{align*}
Thus $X_a$ is again completely determined by invariants of the metric. And now, again as in the Ricci soliton case, plugging this formula for $X_a$ back into the equations for homotheties ((\ref{grs0}) with $c_1=0$, $c_2=0$), gives us local obstructions for the metric to admit a solution to the homothety equation.

Alternatively, we can contract (\ref{pawhom}) by $\epsilon^{ab}\nd_bK$ like before and get
\[
X^bN_b=-\frac{1}{2}MF. 
\]
Hence if $\rho=0$ and $\nu$ is non-zero we can still solve for 
\begin{align*}
X_c=-\frac{1}{\nu}\left(-\frac{1}{2}MF\right)\epsilon_{cd}\nd^dK-\frac{2 K \la}{\nu}\epsilon_{cd}N^d.
\end{align*}
In the Riemannian setting, if $\nu=0$, then $N_a=0$ as before in the proof of Theorem \ref{ricsolthm}, so that $F=0$ and $X$ is again a gradient.  
\end{proof}

\begin{corollary}
The local obstruction for a 2-dimensional (pseudo)-Riemannian metric g with $\rho \neq 0$ to admit a homothety $(g, X)$ is the obstruction tensor $\Theta^{(2)}_{ab}$ given by
\[
\Theta^{(2)}_{ab}=\nd_{(a} X^{(2)}_{b)}-\la g_{ab},
\] 
where 
\begin{align*}
X^{(2)}_a=\frac{1}{\rho}\left(6 M \la \epsilon_{ab}\nd^b K-2 K\la\epsilon_{ab}\nd^b M\right).
\end{align*}
The obstruction tensor $\Theta^{(2)}_{ab}$ vanishes if and only if
\[
(g,X)=(g,X^{(2)})
\]
with $\rho\neq 0$ is a homothety. 
\end{corollary}
\begin{remark} 
Again we observe that in the case $\rho=0$ and $\nu \neq0$, the formula for $X_a$ obtained in the proof of Theorem \ref{homothetythm} given by
\begin{equation}
X_c=\frac{MF}{2 \nu}\epsilon_{cd}\nd^dK-\frac{2 K \la}{\nu}\epsilon_{cd}N^d,\label{xc}
\end{equation}
still involves the unknown quantity $F$. To solve for $F$, we can proceed as in the Remark \ref{fremark}, and in the end obtain local if-and-only-if obstructions for $g$ with $\rho=0$, $\nu \neq 0$ to admit a homothety. Let us show this with an example. 
\end{remark}

\begin{example}

We again consider the metric
\begin{equation}
g={\rm e}^{2f}(\der x^2 +\der y^2)\label{mete},\end{equation}
with $f=f(x)$, as in Example \ref{soi}. 

We now find $f$s for which $g$ has homotheties. As before we have: $\rho\equiv 0$ and $\nu=-{\rm e}^{-10f}f'(2f'f''-f^{(3)})^3$, so if $f'(2f'f''-f^{(3)})\neq 0$ the homothety necessarily has the form (\ref{xc}). 
Explicitly, (\ref{xc}) gives:
\begin{equation} X=\frac{2\lambda f''}{2f'f''-f^{(3)}}~\partial_x~ +~\frac{F}{2f'}~\partial_y.\label{roiz}\end{equation}
Here $F$ is the unknown function $F=F(x,y)$ responsible for the skew symmetric part of $\nabla_a X_b$. We now insert this $X$ to the last two of the closed system equations (\ref{homeq}) and solve for the derivatives $F_x$ and $F_y$. We get:
$$F_x=\frac{f''}{f'}F,\quad\quad F_y=\frac{4\lambda {f''}^2}{f^{(3)}-2f'f''}.$$
Requiring that $F_{xy}=F_{yx}$ we get:
$$4\lambda f''(f'f''f^{(4)}-2f'{f^{(3)}}^2+(2{f'}^2f''+{f''}^2)f^{(3)}-4f'{f''}^3)=0.$$
Now there are three cases. 

The first one, $f''\equiv 0$, corresponds to flat metrics, $K\equiv 0$, and we will not comment on it anymore.

The second case is when 
$$\lambda=0.$$
As such, this case corresponds to $X_a$ which is a Killing vector. In this case we have $F_y=0$, and by integration we get that $F=2a f'$, where $a$ is a constant. This, when inserted in (\ref{roiz}) gives $X=a \partial_y$, which is obviously a Killing vector for metric (\ref{mete}), regardless of what the function $f=f(x)$ is. 

We are left with the analysis of the third case, which requires that 
\begin{equation}(f'f''f^{(4)}-2f'{f^{(3)}}^2+(2{f'}^2f''+{f''}^2)f^{(3)}-4f'{f''}^3)=0.\label{psy}\end{equation}
In this case we insert (\ref{roiz}) in the first four closed system equations (\ref{homeq}). By using the computed $F_x$, $F_y$, and their integrability condition (\ref{psy}), which gives us the $4^{\rm th}$ derivative of $f$, we see that these four equations reduce to a single one, which when $\lambda\neq 0$, is equivalent to 
\begin{equation}
\left(\frac{f''}{{f'}^2}\right)'=0.\label{czw1}\end{equation} 
This equation can be explicitly solved. It follows that its solutions automatically solve the 4th order ODE (\ref{psy}) and as such lead to the metrics (\ref{mete}) admitting homotheties. Solving for $F$ we get the most general solution. Modulo a redefinition of coordinates this most general solution is given by the following proposition.
\end{example}

\begin{proposition}
Modulo a change of coordinates $(x,y)\mapsto(\alpha x+\beta,y)$, a nonflat metric
$$g={\rm e}^{2f}(\der x^2 +\der y^2),$$
with $f=f(x)$, admits a proper homothetic vector $X_a$ if and only if
$$\boxed{g=x^{2s}(\der x^2+\der y^2),}$$
with constant $s$ such that $s(s+1)\neq 0$. In such case homothetic vector fields are given by a 2-parameter family:
$$\boxed{X=b ( x\partial_x+y\partial_y)+a\partial_y,}$$
parametrised by constants $a$ and $b$. The parameter $b$ is related to the expansion of $X$, via
$$\nabla_{(a}X_{b)}=(s+1)b g_{ab}.$$
\end{proposition}

\subsection{2D Killing equations}{\bf The case where $c_1=0$, $c_2=0$, $\la=0$.}
The Killing equation has been studied extensively and much is known in the classical literature. 
The vanishing of $\rho=\epsilon^{ab}\nd_a K \nd_b M$ is necessary for the metric to admit an isometry and is known by geometers such as Liouville and Darboux as $I_1$. 
The quantity $\rho$ coincides with the projective invariant $\rho$ in \cite{skewricci2D} up to some non-zero multiple. It is also known as $\nu_5$ by Liouville. There is another invariant known as $I_2$, which is given by non-zero multiple of $\epsilon^{ab}\nd_aK \nd_b \De K$, whose vanishing together with $I_1$ characterise metrics that are locally surfaces of revolution. See \cite{darboux}, \cite{dunos}, \cite{krug} and \cite{cubic} for details.

\subsection{2D metric projective structures with skew-symmetric Ricci tensor}{\bf The case where $c_1=1$, $c_2=-1$, $\la=0$.}
For this section we give just the closed system and integrability conditions for metric projective Einstein-Weyl (pEW) equations. The pEW equation is introduced in \cite{thesis} and \cite{skewricci2D}. Local obstructions (at least in the Riemannian setting) have already been found in \cite{skewricci2D}. The 2 dimensional projective Einstein-Weyl equations obtained from setting $c_1=1$, $c_2=-1$, $\la=0$ in the generalised Ricci soliton equations are 
\begin{equation}\label{pEW}
\nd_{a}X_{b}+X_aX_b+K g_{ab} =F_{ab}, 
\end{equation}
with $F_{ab}=\frac{e}{2}\epsilon_{ab}F=\nd_{[a}X_{b]}$ where $F^{ab}\epsilon_{ab}=F$. (Note that $\epsilon_{ac}\epsilon^{bc}=e\de^b{}_c$.)

We prolong to get the closed system
\begin{equation}\label{pEW2}
\boxed{
\begin{aligned}
\nd_{a}X_{b}&=-X_aX_b -K g_{ab}+\frac{e}{2}F \epsilon_{ab},\\
\nd_aF&=-3X_a F-2\epsilon^b{}_a\nd_bK.
\end{aligned}}
\end{equation}
The constraint equation is given by
\begin{equation}\label{pEWcons}
\boxed{
\begin{aligned}
X^a\nd_aK=-\frac{e}{2}F^2-\frac{\De K}{3}.
\end{aligned}}
\end{equation}

\subsection{2D near-horizon geometry equations}{\bf The case where $c_1=1$, $c_2=\frac{1}{2}$, $\la=0$.}
For this section we just show the computations that lead to algebraic constraints for the near-horzion geometry equations in 2 dimensions. The near-horizon geometry equations in 2 dimensions are given by
\begin{equation}\label{gs2p}
\nd_{a}X_{b}+X_aX_b-\frac{K}{2}g_{ab} =F_{ab}, 
\end{equation}
where again $F_{ab}=\frac{e}{2}\epsilon_{ab}F=\nd_{[a}X_{b]}$ with $F^{ab}\epsilon_{ab}=F$.
We prolong to get the closed system
\begin{equation}\label{nhg2}
\boxed{
\begin{aligned}
\nd_{a}X_{b}&=-X_aX_b +\frac{K}{2} g_{ab}+\frac{e}{2}F \epsilon_{ab},\\
\nd_aF&=-3X_a F+\epsilon^b{}_a\nd_bK+3\epsilon^b{}_aX_b K.
\end{aligned}}
\end{equation}
The integrability condition is given by
\begin{equation*}
\boxed{
\begin{aligned}
X^a\nd_aK=\frac{e}{2}F^2-X_aX^a K-\left(\frac{\De K}{6}+\frac{K^2}{2}\right).
\end{aligned}}
\end{equation*}
We call $X_aX^a=\si$ and $\mu=\frac{\De K}{6}+\frac{K^2}{2}$,
so that the constraint equation is now 
\begin{equation}\label{nhgcons}
X^a\nd_aK=\frac{e}{2}F^2-\si K-\mu.
\end{equation}
Observe that as a consequence of the first equation in (\ref{nhg2}), we have
\begin{align*}
\nd_a \si=e\epsilon_{ab}X^b F+K X_a-2 X_a \si. 
\end{align*}
Differentiating (\ref{nhgcons}) one more time, and using (\ref{nhg2}), we obtain
 \begin{align*}
\left(\frac{eF}{2}\epsilon_a{}^b+\frac{K}{2}\de_a{}^b-X_aX^b\right)&\nd_b K+X^b(\nd_a\nd_bK)\\
=&-3eX_a F^2+e\epsilon^b{}_a(\nd_b K)F+3 e\epsilon^b{}_aX_b KF\\
&-e\epsilon_{ab}X^b KF-K^2X_a+2 X_a \si K-\si (\nd_aK)-\nd_a\mu.
\end{align*}
Now contract the above equation with $\nd^aK$, and define $M:=\nd_aK \nd^aK$. We get 
\begin{align}\label{nhgm}
\frac{KM}{2}-&(X^a\nd_a K)(X^b\nd_b K)+\frac{1}{2}X^b\nd_bM\nonumber\\
=&-3e(X^a\nd_aK) F^2-4e\epsilon^{ab}\nd_a K X_b KF-K^2(X^a\nd_a K)\nonumber\\
&+2 (X^a\nd_aK) \si K-\si M-\nd^aK \nd_a\mu.
\end{align}
Let us define 
\[
A:=\frac{e}{2}F^2-\si K-\mu,
\]
so that (\ref{nhgcons}) is now
\[
X^a\nd_a K=A. 
\]
Note that $A$ depends on unknowns $\si$, $F$ and metric invariants $K$, $\mu$.
Equation (\ref{nhgm}) becomes
 \begin{align*}
X^b\big(\nd_bM+&8e\epsilon^{a}{}_b(\nd_a K) KF\big)\\
=&2A^2-KM-6eA F^2-2A K^2+4 AK \si-2\si M-2\nd^aK \nd_a\mu.
\end{align*}
 Now define
 \[
 B:=2A^2-KM-6eA F^2-2A K^2+4 AK \si-2\si M-2\nd^aK \nd_a\mu.
 \]
Again $B$ depends on unknowns $\si$, $F$ and metric invariants $K$, $\mu$, $M$, $\nd^aK \nd_a\mu$,
so that
\begin{align*}
X^b\nd_bM+8e\epsilon^{ab}\nd_a K X_b KF=B.
\end{align*}
We can now express $X_a$ in the basis $\epsilon_{ab}\nd^b K$ and $\epsilon_{ab}\nd^b M$, assuming 
\[
\rho=\epsilon^{ab}\nd_aK \nd_b M \neq 0. 
\]
We obtain 
 \begin{align}\label{xa}
 X_a=-\frac{1}{\rho}(B-8e\epsilon^{cd}(\nd_c K) X_d KF)\epsilon_{ab}\nd^bK+\frac{A}{\rho}\epsilon_{ab}\nd^b M.
 \end{align}
Observe that the term $X_d$ still appears on the right hand side of the above expression, which we now eliminate. Contracting with $\epsilon^{ca}\nd_c K$, we get 
  \begin{align*}
 \epsilon^{ca}\nd_c KX_a=\frac{eM}{\rho}(B-8e\epsilon^{ab}\nd_a K X_b KF)-\frac{e A}{\rho}\nd_bK\nd^b M.
 \end{align*}
 Define $N:=\nd_aK \nd^a M$. We therefore have 
  \begin{align*}
(\rho+8KFM)\epsilon^{ab}\nd_a KX_b=e(MB-AN),
 \end{align*}
 from which we get
   \begin{align}\label{mban}
 \epsilon^{ab}\nd_a KX_b=\frac{e(MB-AN)}{\rho+8KFM}.
 \end{align}
Now assuming $\rho+8KFM \neq 0$, we can substitute (\ref{mban}) back into the expression for $X_a$ given by (\ref{xa}), but also observe that $\nd_aK$ and $\epsilon_{ab}\nd^b K$ constitute a new basis. Assuming that $M\neq 0$, we therefore obtain 
 \begin{align}\label{xsig}
X_a=\frac{AN-MB}{M(\rho+8KFM)}\epsilon_{ab}\nd^bK+\frac{A}{M}\nd_a K.
 \end{align}
We therefore obtained an expression for $X$, provided that the near-horizon geometry equation is satisfied, involving unknowns $F$, $\si$ and metric invariants $K$, $\mu$, $M$, $\nd^aK \nd_a\mu$, $N$.
Next, we try to eliminate $\si$. We find that
  \begin{align*}
 \si=X_aX^a
  =&\frac{e(AN-MB)^2}{M(\rho+8KFM)^2}+\frac{A^2}{M}.
 \end{align*}
 Let us call $J=\rho+8KFM$. Rearranging the terms, we find that
 \begin{equation}\label{nhgquartic}
 M J^2\si=e(AN-MB)^2+A^2J^2, 
 \end{equation}
 and since $A$ is linear in $\si$, $B$ is quadratic in $\si$, this means that equation (\ref{nhgquartic}) gives us a quartic polynomial equation that $\si$ has to satisfy, with coefficients of the polynomial given by expressions involving differential invariants of the metric and also the unknown $F$. Since quartic polynomials are solvable, we can solve for $\si$ (in terms of metric invariants and $F$) and plug this expression for $\si$ back into (\ref{xsig}) to get $X$ determined now only in terms of $F$ and its metric invariants. We can then substitute $X$ back into original equation and derive further algebraic constraints. 
To summarise, we have:
\begin{proposition}\label{nhgprop}
A 1-form $X$ defining a generalised Ricci soliton $(g,X)$ with $(c_1,c_2,\la)=(1,\frac{1}{2},0)$ is of the form
\begin{align*}
X_a=\frac{AN-MB}{M(\rho+8KFM)}\epsilon_{ab}\nd^bK+\frac{A}{M}\nd_a K,
\end{align*}
provided $J=\rho+8KFM \neq 0$, $M \neq 0$. Here $A$, $B$ are quantities as defined above involving the unknowns $\si=X_aX^a$, $F$ and metric invariants. 
Moreover, $\si$ must satisfy a quartic polynomial equation
\[
 M J^2\si=e(AN-MB)^2+A^2J^2 
\]
still involving the unknown $F$ and the metric invariants that appear in $M$, $J$, $A$, $B$, $N$. 
\end{proposition}
Because of the tedious and difficult nature of the computations we shall not proceed further.   
Finally, note that in the case where $J=\rho+8KFM=0$, we have $F=-\frac{\rho}{8KM}$ and also $MB=AN$. Equation (\ref{nhgquartic}) becomes an identity.

Let us now turn to gradient generalised Ricci solitons in higher dimensions.

\section{Gradient generalised Ricci solitons}{\bf The case where $F_{ab}=0$ in $n$ dimensions.}
In this section we generalise the results of Section \ref{ggrs2d} on 2-dimensional gradient generalised Ricci soliton to arbitrary dimensions. 
Such a soliton has the vector field $X_a$ that is locally a gradient. Therefore in this section all our considerations are about the case when 
$$F_{ab}\equiv 0$$ 
in the closed system (\ref{gspr}). 

\begin{proposition}\label{gradientpropn}
Let $(M,g)$ be a (pseudo)-Riemannian $n$-manifold. Let $R_{ab}$ be the Ricci tensor and $R$ be the Ricci scalar curvature for $g$. Assume that 
\[
\rho^{a}{}_b:=(1-c_1c_2)R^{a}{}_b+\left((n-1)\la c_1+c_1c_2 R\right)\de^{a}{}_b
\] 
has non-zero determinant. Then a necessary condition for $X$ to be a gradient generalised Ricci soliton $(g, X)$ is that
\begin{align}\label{grsgradx}
X_a=-\frac{c_2}{2}\tilde \rho_a{}^b\nd_bR,
\end{align}
where the symbol $\tilde \rho^a{}_b$ denotes the matrix inverse of $\rho^a{}_b$. 
The vector field $X_a$ is a gradient generalised Ricci soliton if and only if it further satisfies
\begin{align}\label{grsgradn}
\nd_{a}X_{b}=&-c_1 X_aX_b +c_2 R_{ab}+ \la g_{ab}.
\end{align} 
\end{proposition}
As a Corollary to Proposition \ref{gradientpropn}, we have:
\begin{corollary}
The local obstructions for a $n$-dimensional (pseudo)-Riemannian metric $g$ with $\rho^a{}_b$ invertible to admit a gradient generalised Ricci soliton $(g,X)$ is the obstruction tensor $\Theta^{(3)}_{ab}$ given by
\[
\Theta^{(3)}_{ab}=\nd_aX^{(3)}_b+c_1X^{(3)}_aX^{(3)}_b-c_2R_{ab}-\la g_{ab}
\]
where
\[
X^{(3)}_a=-\frac{c_2}{2}\tilde \rho_a{}^b\nd_bR.
\]
The obstruction tensor $\Theta^{(3)}_{ab}$ vanishes if and only if $(g,X)=(g,X^{(3)})$ is a gradient generalised Ricci soliton.
\end{corollary}
\begin{proof}[Proof of Proposition \ref{gradientpropn}]
If we look at the equations in the closed system (\ref{gspr})
and consider the case $F_{ab}=0$, we get from the second that 
\begin{align*}
0=&c_2(\nd_{b}R_{ca}-\nd_c R_{ba})+R_{bc}{}^d{}_aX_d\\
&+\la c_1 X_b g_{ca}- \la c_1 X_c g_{ba}+c_1c_2 X_b R_{ca}-c_1 c_2 X_c R_{ba},
\end{align*}
from which, tracing $c$ and $a$ indices, gives
\begin{align*}
0=&c_2(\nd_{b}R-\nd^c R_{ca})+R_{bd}X^d+(n-1)\la c_1 X_b+c_1c_2 X_b R-c_1 c_2 R_{bd}X^d.
\end{align*}
Rearranging, and using the contracted Bianchi identity that $\nd_aR=2\nd^bR_{ab}$, we obtain
\begin{align*}
0=&\frac{c_2}{2}\nd_{b}R+(1-c_1c_2)R_{bd}X^d+\left((n-1)\la c_1+c_1c_2 R\right) X_b,
\end{align*}
or that
\begin{align*}
-\frac{c_2}{2}\nd^{b}R=\left[(1-c_1c_2)R^b{}_{d}+\left((n-1)\la c_1+c_1c_2 R\right)\de^b{}_d\right] X^d.
\end{align*}
Hence, supposing that the tensor given by
\[
\rho^a{}_{b}:=(1-c_1c_2)R^a{}_{b}+\left((n-1)\la c_1+c_1c_2 R\right)\de^a{}_{b}
\]
has non-zero determinant, we have its inverse $\tilde \rho^a{}_{b}$ such that
\[
\tilde \rho^a{}_b \rho^{b}{}_c=\de^a{}_c.
\]
In this case, we solve for $X^a$, obtaining
\begin{align*}
X^a=-\frac{c_2}{2}\tilde \rho^a{}_b\nd^bR.
\end{align*}
This proves that $X_a$ is of the form (\ref{grsgradx}). To make this necessary condition  for $X_a$ sufficient, vector field $X^a$ must satisfy the equation (\ref{grsgradn}). This ends the proof.
\end{proof}
\begin{remark}
Although it is not evident at the first glance, 
it follows from the proof, as a consequence of $F_{ab}=0$, that the vector field $X^a$ given by formula (\ref{grsgradx}) is locally a gradient.
\end{remark}
\begin{example}[Gradient Ricci solitons: $c_1=0$, $c_2=-1$]
We see that in this case, $\rho_{ab}=R_{ab}$. For example, can the metric in $\R^3$ given by 
\[
g=e^{t^2} \der x^2+ e^{t} \der y^2+\der t^2
\]
admit a steady $(\la=0)$ gradient Ricci soliton?
We find for this metric that the Ricci tensor is given by 
$$
R_{ab}\der x^a \der x^b=\frac{e^{t^2}(2t^2+t+2)}{2} \der x^2-\frac{e^t(1+2t)}{4}\der y^2-\left(t^2+\frac{5}{4}\right)\der t^2
$$
and therefore is invertible on the open set where $(2t^2+t+2)(1+2t)$ is non-zero.
We compute, and find that $X_a=\frac{1}{2}\tilde \rho_a{}^b\nd_bR$ gives
\[
X=-\frac{2(4t+1)}{4t^2+5}\der t.
\]
Plugging this back into the steady gradient Ricci soliton equation gives
$$\begin{aligned}
\Theta^{(3)}_{ab}&\der x^a \der x^b=\frac{e^{t^2}(8t^4+4t^3+2t^2+t+10)}{2(4t^2+5)}\der x^2\\&+\frac{e^{t^2}(8t^3+4t^2-6t+1)}{4(4t^2+5)}\der y^2+\frac{64 t^6+240 t^4+428 t^2+64 t-35}{4(4t^2+5)^2}\der t^2,
\end{aligned}$$
which is non-zero. We conclude that this metric does not admit a solution to the steady gradient Ricci soliton equations even locally. 
\end{example}

\begin{example}[Gradient Ricci solitons: $c_1=0$, $c_2=-1$]\label{3dricsol}
For a positive example, consider the metric on $\R^3$ given by 
\[
g=t \der x^2+ t \der y^2+\frac{2 t^{\sqrt{2}-2}}{(a t^{\sqrt{2}}-b)^2}\der t^2.
\]
We find for this metric that the Ricci tensor is given by 
\begin{align*}
R_{ab}\der x^a \der x^b=&-\frac{a^2(2+\sqrt{2})t^{2+\sqrt{2}}-4 a b t^2-b^2(\sqrt{2}-2)t^{2-\sqrt{2}}}{8t}\der x^2\\
&-\frac{a^2(2+\sqrt{2})t^{2+\sqrt{2}}-4 a b t^2-b^2(\sqrt{2}-2)t^{2-\sqrt{2}}}{8t}\der y^2\\
&-\frac{a(\sqrt{2}+1) t^{\sqrt{2}}+b(\sqrt{2}-1)}{2(a t^{\sqrt{2}}-b)t^2}\der t^2
\end{align*}
and therefore is invertible on the open set where 
\[
(a^2(2+\sqrt{2})t^{2+\sqrt{2}}-4 a b t^2-b^2(\sqrt{2}-2)t^{2-\sqrt{2}}) \left(a(\sqrt{2}+1) t^{\sqrt{2}}+b(\sqrt{2}-1)\right)
\]
is non-zero.
We compute, and find that $X_a=\frac{1}{2}\tilde \rho_a{}^b\nd_bR$ gives
\[
X=\frac{\sqrt{2}(a^2(2\sqrt{2}+3)t^{2+2\sqrt{2}}+b^2(2\sqrt{2}-3)t^2)}{2(a t^{\sqrt{2}}-b)(a(\sqrt{2}+1)t^{\sqrt{2}}+b(\sqrt{2}-1))t^3}\der t.
\]
Plugging this back into the steady gradient Ricci soliton equation gives
$$\Theta^{(3)}_{ab}\der x^a \der x^b=0.$$
We therefore conclude that 
$$
\boxed{\begin{aligned}
g&=t \der x^2+ t \der y^2+\frac{2 t^{\sqrt{2}-2}}{(a t^{\sqrt{2}}-b)^2}\der t^2,\\
X&=\frac{\sqrt{2}(a^2(2\sqrt{2}+3)t^{2+2\sqrt{2}}+b^2(2\sqrt{2}-3)t^2)}{2(a t^{\sqrt{2}}-b)(a(\sqrt{2}+1)t^{\sqrt{2}}+b(\sqrt{2}-1))t^3}\der t,
\end{aligned}}
$$
is a 2-parameter family of steady gradient Ricci solitons. 
\end{example}

\begin{remark}
In the case where $1=c_1c_2$, we see that $\rho_{ab}$ is a multiple of the metric and so $X_a$ will be some multiple of the gradient of $R$.
\end{remark}

\section{3D Ricci solitons and homotheties}
Again because of the difficulty in considering the $c_1 \neq 0$ case, let us now consider the case for the generalised Ricci solitons with $c_1=0$ (proper Ricci solitons and homotheties) in 3 dimensions.

The prolongation gives
\begin{equation}\label{3drss}
\boxed{\begin{aligned}
\nd_{a}X_{b}=&c_2 R_{ab}+F_{ab}+ \la g_{ab}\\
\nd_aF_{bc}=&c_2(\nd_{b}R_{ca}-\nd_c R_{ba})+R_{bc}{}^d{}_aX_d
\end{aligned}}
\end{equation}
Tracing over $a$ and $b$ indices in the second equation and using the contracted Bianchi identity gives
\begin{align*}
\nd^aF_{ac}
=&-\frac{c_2}{2}\nd_c R-R_{cd}X^d. 
\end{align*}
Differentiating this equation once more and using (\ref{3drss}) yields
\begin{align*}
\nd_b\nd^aF_{ac}
=&-\frac{c_2}{2}\nd_b\nd_c R-(\nd_b R_{cd})X^d-R_{cd}(c_2 R_b{}^d+F_b{}^d+\la \de_b{}^d).
\end{align*}
Contracting upon $b$ and $c$ and using the identity $\nd^a\nd^bF_{ab}=0$, which is true for any $2$-form by the Ricci identity, we obtain
\begin{equation}\label{3drsc}
\boxed{
\begin{aligned}
X^d\nd_d R=-c_2\De R-2c_2 R_{bd}R^{bd}-2\la R
\end{aligned}}
\end{equation}
as an integrability condition to (\ref{3drss}). To get further algebraic constraints, we can differentiate (\ref{3drsc}).

Differentiating the last equation in (\ref{3drss}) gives
\begin{align*}
\nd_d\nd_aF_{bc}=c_2(\nd_d\nd_bR_{ca}-\nd_d\nd_cR_{ba})+(\nd_dR_{bc}{}^e{}_a)X_e+R_{bc}{}^e{}_a(c_2R_{de}+F_{de}+\la g_{de}), 
\end{align*}
so that skewing gives
\begin{align*}
R_{ad}{}^e{}_bF_{ec}+R_{ad}{}^e{}_cF_{be}=&c_2(\nd_d\nd_bR_{ca}-\nd_d\nd_cR_{ba})-c_2(\nd_a\nd_bR_{cd}-\nd_a\nd_cR_{bd})\\
&+(\nd_dR_{bc}{}^e{}_a-\nd_aR_{bc}{}^e{}_d)X_e+c_2 R_{bc}{}^e{}_aR_{de}+R_{bc}{}^e{}_aF_{de}+\la R_{bcda}\\
&-c_2 R_{bc}{}^e{}_dR_{ae}-R_{bc}{}^e{}_dF_{ae}-\la R_{bcad}.
\end{align*}
Using the Bianchi identities, 
\[
\nd_d R_{bc}{}^e{}_a-\nd_a R_{bc}{}^e{}_d=\nd_d R^e{}_{abc}-\nd_a R^{e}{}_{dbc}=\nd^e R_{dabc},
\]
and tracing $d$ and $b$ indices, we obtain
\begin{align*}
&X^e\nd_eR_{ac}+R^b{}_c{}^e{}_aF_{be}+R_c{}^eF_{ae}
+R^b{}_a{}^e{}_cF_{be}+R_a{}^eF_{ce}\\
=&c_2(-R_c{}^eR_{ae}-R^b{}_c{}^e{}_aR_{be})-2\la R_{ca}+c_2(\nd_a\nd^bR_{cb}-\nd_a\nd_cR)
-c_2(\De R_{ca}-\nd^b\nd_cR_{ba}).
\end{align*}
Let us denote the right hand side by $S_{ac}$: 
\[
S_{ac}=c_2(-R_c{}^eR_{ae}-R^b{}_c{}^e{}_aR_{be})-2\la R_{ca}+c_2(\nd_a\nd^bR_{cb}-\nd_a\nd_cR)
-c_2(\De R_{ca}-\nd^b\nd_cR_{ba}).
\]
Let us write 
\[
R_{ab}=\P_{ab}+\P g_{ab}.
\]
Here $\P_{ab}$ is the Schouten tensor and $\P=g^{ab}\P_{ab}$ is its metric trace.
Decomposing 
\[
R_{abcd}=g_{ac}\P_{bd}-g_{bc}\P_{ad}+g_{bd}\P_{ac}-g_{bc}\P_{ad}, 
\]
we then get
\begin{align*}
&X^e\nd_eR_{ac}+(g^{be}\P_{ca}-\de_c{}^e\P_a{}^b+\P^{be}g_{ca}-\de_a{}^b\P_c{}^e)F_{be}
+R_c{}^eF_{ae}\\
&+(g^{be}\P_{ac}-\de_a{}^e\P_c{}^b+\P^{be}g_{ca}-\de_c{}^b\P_a{}^e)F_{be}+R_a{}^eF_{ce}\\
=&X^e\nd_eR_{ac}-\P_a{}^bF_{bc}-\P_c{}^bF_{ba}\\
=&S_{ac}.
\end{align*}
Let $\epsilon_{abc}$ denote the volume form in $3$ dimensions. Again setting $e=1$ when $g$ is Riemannian $(+++)$ and $e=-1$ when $g$ is Lorentzian $(++-)$, we have
\begin{align*}
\epsilon_{abc}\epsilon_{def}
=&e\left(g_{ad}g_{be}g_{cf}+g_{af}g_{bd}g_{ce}+g_{ae}g_{bf}g_{cd}-g_{ad}g_{bf}g_{ce}-g_{ae}g_{bd}g_{cf}-g_{af}g_{be}g_{cd}\right),
\end{align*}
from which we get
\begin{align*}
\epsilon_{abc}\epsilon_{de}{}^c=e(g_{ad}g_{be}-g_{ae}g_{bd}).
\end{align*}
We can use the volume form to dualise the 2-form $F_{ab}$, so that 
\[
F_{ab}=\frac{e}{2}\epsilon_{abc}F^c
\]
where $F^a=\epsilon^{abc}F_{bc}$.
We can therefore write
\begin{align}\label{3constraint}
X^e\nd_eR_{ac}=&S_{ac}+\P_c{}^bF_{ba}+\P_a{}^bF_{bc}\nonumber\\
=&S_{ac}+\frac{e}{2}\epsilon_{bad}\P_c{}^bF^d+\frac{e}{2}\epsilon_{bcd}\P^b{}_aF^d\nonumber\\
=&S_{ac}-e\P^b{}_{(c}\epsilon_{a)bd}F^d.
\end{align}
Tracing the $a$ and $c$ indices gives the integrability condition (\ref{3drsc}).
Now let us call 
\begin{align*}
E_{(ca)d}=\P^b{}_{(c}\epsilon_{a)bd}.
\end{align*}
A computation shows that
\begin{align*}
E^{(ca)}{}_bE_{(ca)d}=&\left(\frac{1}{2}\P^{ea}\epsilon^c{}_{eb}+\frac{1}{2}\P^{ec}\epsilon^a{}_{eb}\right)\left(\frac{1}{2}\P^{f}{}_{a}\epsilon_{cfd}+\frac{1}{2}\P^{f}{}_{c}\epsilon_{afd}\right)\\\\
=&\frac{1}{2}\P^{ea}\P^f{}_a \epsilon^c{}_{eb}\epsilon_{cfd}
+\frac{1}{2}\P^{ea}\P^f{}_c \epsilon^c{}_{eb}\epsilon_{afd}\\
=&\frac{e}{2}\P^{ea}\P^f{}_a\left(g_{ef}g_{bd}-g_{ed}g_{bf}\right)+\frac{e}{2}\P^{ea}\P^f{}_c\left(\de^c{}_ag_{ef}g_{bd}+\de^c{}_dg_{ea}g_{bf}\right.\\
&\left.+\de^c{}_fg_{ed}g_{ba}
-\de^c{}_ag_{bf}g_{ed}-\de^c{}_dg_{ba}g_{ef}
-\de^c{}_fg_{bd}g_{ea}\right)\\
=&e (\P_{ac}\P^{ac}g_{bd}+\P\P_{bd}-\frac{3}{2}\P_{b}{}^e\P_{de}-\frac{1}{2}\P^2 g_{bd})\\
=&eQ_{bd}.
\end{align*}
Contracting (\ref{3constraint}) throughout with $E^{(ca)}{}_{b}$, we obtain
\begin{align*}
(X^e\nd_eR_{ac})E^{(ac)}{}_b
=&S_{ac}E^{(ab)}{}_b-Q_{bd}F^d.
\end{align*}

Hence provided that $Q^b{}_{d}$ is invertible, (so that there exists $\tilde Q^{a}{}_c$ such that $\tilde Q^{a}{}_cQ^c{}_b=\de^a{}_b$), we can solve for $F_d$ in terms of $X_a$ and invariants of the (pseudo)-Riemannian structure, to obtain
\begin{align}\label{3dfx}
F_a=\tilde Q_{ab}(S_{dc}E^{(dc)b}-(X^e\nd_eR_{dc})E^{(dc)b}).
\end{align}
In other words there is an algebraic relation between $X_a$ and $F_a$ in 3 dimensions. Plugging the expression for $F_a$ given by (\ref{3dfx})
back into the differential constraint
\[
X^e\nd_eR_{ac}=S_{ac}-eE_{(ca)d}F^d
\]
gives
\begin{align*}
X^e\nd_eR_{ac}=S_{ac}-eE_{(ca)d}\tilde Q^d{}_{b}(S_{fk}E^{(fk)b}-X^e(\nd_e R_{fk})E^{(fk)b}),
\end{align*}
so that
\begin{align*}
X^e\nd_eR_{ac}-X^e(\nd_e R_{fk})eE_{(ca)d}\tilde Q^d{}_{b}E^{(fk)b}=S_{ac}-eE_{(ca)d}\tilde Q^d{}_{b}S_{fk}E^{(fk)b},
\end{align*}
which implies
\begin{equation*}
X^e(\nd_eR_{ac}-e(\nd_e R_{fk})E_{(ca)d}\tilde Q^d{}_{b}E^{(fk)b})=S_{ac}-eE_{(ca)d}\tilde Q^d{}_{b}S_{fk}E^{(fk)b}.
\end{equation*}
This is not identically zero, and we have eliminated $F^a$ without differentiating (\ref{3dfx}) further. Also, by taking the metric trace, we obtain the constraint (\ref{3drsc}).
To summarise, we have:
\begin{proposition}\label{3dprop}
Let $(M,g)$ be a (pseudo)-Riemannian $3$-manifold. Let $\P_{ab}$ be the Schouten tensor and $\P=g^{ab}\P_{ab}$ be its trace.
Assume that 
\begin{align*}
Q^b{}_{d}=\P_{ac}\P^{ac}\de^b{}_{d}+\P\P^b{}_{d}-\frac{3}{2}\P^{be}\P_{de}-\frac{1}{2}\P^2 \de^b{}_{d}
\end{align*}
has non-zero determinant. Then every $c_1=0$ generalised Ricci soliton $(g, X)$ has
\begin{equation}\label{Prop5.1}
\boxed{
X^e(\nd_eR_{ac}-e(\nd_e R_{fk})E_{(ca)d}\tilde Q^d{}_{b}E^{(fk)b})=S_{ac}-eE_{(ca)d}\tilde Q^d{}_{b}S_{fk}E^{(fk)b}.}
\end{equation}
In formula (\ref{Prop5.1}) the symbol $\tilde Q^a{}_b$ denotes the matrix inverse of $Q^a{}_b$.
We also have
\[
S_{ac}=c_2(-R_c{}^eR_{ae}-R^b{}_c{}^e{}_aR_{be})-2\la R_{ca}+c_2(\nd_a\nd^bR_{cb}-\nd_a\nd_cR)
-c_2(\De R_{ca}-\nd^b\nd_cR_{ba})
\]
and
\[
E_{(cd)b}=\P^a{}_{(c}\epsilon_{d)ab}.
\]
\end{proposition}

For certain examples constraint (\ref{Prop5.1}) is already sufficient to conclude that locally some metrics cannot admit any solution to the generalised Ricci soliton equations with $c_1=0$. Instead of a tensorial obstruction, we obtain algebraic relations that force a linear system of equations to be inconsistent. To illustrate this, let us take the metric given by
\[
g=e^{z} \der x^2+e^{-z} \der y^2+z \der z^2.
\]
We ask: can this metric admit a homothetic vector field that is not Killing $(c_1=0, c_2=0, \la \neq 0)$?
For generic $c_2$ and $\la$, we compute and find that 
\[
Q_{ab}\der x^a\der x^b=\frac{e^z(4z^2-4z+1)}{32 z^4}\der x^2+\frac{e^{-z}(4 z^2+4z+1)}{32 z^4}\der y^2+\frac{1}{8 z^3} \der z^2, 
\]
and this is invertible on an open set away from $(2z-1)^2(2z+1)^2=0$.
We obtain
\[
\tilde Q_{ab}\der x^a\der x^b=\frac{32 z^4 e^{z}}{4z^2-4z+1}\der x^2+\frac{32 e^{-z}z^4}{4 z^2+4z+1}\der y^2+8 z^5 \der z^2. 
\]
We find that $S_{ab}E^{(ab)c}=0$, so the right hand side expression of (\ref{Prop5.1}) reduces to 
\begin{align*}
S_{ac}\der x^a \der x^c=&\frac{e^z(-2 \la z^3+2 c_2 z^2-7 c_2)}{4z^5} \der x^2-\frac{e^{-z}(-2\la z^3+2 c_2 z^2-7 c_2)}{4 z^5} \der y^2\\
&-\frac{-2 \la z^3+c_2 z^2-2 c_2}{2 z^3} \der z^2.
\end{align*}
Let 
\[
X=X_1(x,y,z) \partial_x+X_2(x,y,z) \partial_y +X_3(x,y,z) \partial_z
\]
denote the vector field $X$ in (\ref{Prop5.1}).
The left hand side of (\ref{Prop5.1}) gives
\[
-\frac{e^zX_3}{2 z^3}\der x^2+\frac{e^{-z}X_3}{2 z^3}\der y^2+\frac{X_3}{2z}\der z^2.
\]
Equating both sides, we obtain the constraint that:
\begin{align}\label{consthom}
&-\frac{e^z(-2 \la z^3+2 X_3 z^2+2 c_2 z^2-7 c_2)}{4z^5}\der x^2+\frac{e^{-z}(-2 \la z^3+2 X_3 z^2+2 c_2 z^2-7 c_2)}{4z^5}\der y^2\nonumber \\
&+\frac{-2 \la z^3+ X_3 z^2+ c_2 z^2-2 c_2}{2z^3}\der z^2=0.
\end{align}
The metric trace of this equation gives
\[
\frac{-2 \la z^3+X_3 z^2+c_2 z^2-2 c_2}{2 z^4}=0,
\]
which is the scalar constraint (\ref{3drsc}) obtained previously.
This gives
\[
X_3=\frac{2\la z^3-c_2 z^2+2 c_2}{z^2}. 
\]
Substituting this expression for $X_3$ back into the constraint (\ref{consthom}) gives
\[
0=-\frac{e^z (2\la z^3-3 c_2)}{4 z^5}\der x^2+\frac{e^{-z}(2\la z^3-3c_2)}{4 z^5}\der y^2.
\]
For $c_2=0$, this is non-zero unless $\la=0$ also. Hence we conclude that locally this metric cannot admit any homothetic vector field unless $\la=0$. It can be verified that this metric admits Killing symmetries, but solving the equations for homotheties with $\la \neq 0$ gives rise to an inconsistent system.

\section{Rewriting generalised 2D Ricci soliton equations in terms of a potential}
\label{potential}
We now restrict ourselves to the 2-dimensional setting to get explicit examples of generalised Ricci solitons. In this section we focus on the situation where $X_a$ is non-null, i.e. 
$$X_aX^a \neq 0.$$ 
We recall that with our notation as in Section \ref{2DgRsT}, 
the generalised Ricci soliton equations in two dimensions are given by
\begin{align}\label{2dc1=1}
c_1&=1: &\nd_{(a}X_{b)}+X_aX_b-(c_2K+\la)g_{ab}&=0,\\ \label{2dc1=0}
c_1&=0: & \nd_{(a}X_{b)}-(c_2K+\la)g_{ab}&=0.
\end{align}
\subsection{Generalisation of Jezierski's formulation of 2D near-horizon equations}
It was recently noted \cite{stable} in the context of General Relativity that the special case $c_1=1, c_2=\frac{1}{2}$, $\la=0$ of generalised Ricci soliton equations admits a convenient description as follows.

The generalised Ricci soliton equation in this case, 
\begin{align}\label{nhgv2}
\nd_{(a}X_{b)}+X_aX_b-\frac{K}{2}g_{ab}=0,
\end{align}
called by the General Relativity community the basic equation of vacuum near-horizon geometry,
is equivalent to the system of 3 equations:
\begin{equation}
\begin{aligned}
\nd_{a}\Phi^a=&1,\\
\nd_{[a}\Phi_{b]}=&0,\\
\nd_aX^a+X_aX^a-K=&0,
\end{aligned}\label{zej}
\end{equation}
where 
\[
\Phi_a=\frac{X_a}{X_bX^b}.
\]
The second of equations (\ref{zej}) is considered as a condition for a local existence of a potential $V$ such that $\Phi_a=\nabla_aV$. The first equation is then the Poisson equation
$$\triangle V=1$$
for the potential $V$. A solution to the Poisson equation is then a solution to the basic equation of near horizon geometry if and only if the vector field $X_a$ satisfies the last equation (\ref{zej}).

The Poisson system,
\[
\nd_{[a}\Phi_{b]}=0, \qquad \nd_a\Phi^a=1
\]
corresponds to the trace-free part of the basic equation of near horizon geometry. The last equation in the system (\ref{zej}) is its trace.

In \cite{Jezierski} and \cite{stable} an axially symmetric ansatz for $g$ was made, with a particular class of solutions to the  Poisson's equations chosen, such that the last equation (\ref{zej}) reduced to a second order linear ODE on a single function of one variable. See Section \ref{sec:2dex} of the present paper to get our version of this result.

In the rest of this section we give two propositions which extend this `near-horizon geometry' approach to the case of the generalised Ricci soliton equations with general $c_2$, general $\lambda$, and with $c_1=1$ or $0$. 
\begin{proposition}[Poisson]\label{Poisson}
The 2-dimensional generalised Ricci soliton equation with $c_1=1$
\[
\nd_{(a}X_{b)}+X_aX_b-(c_2K+\la)g_{ab}=0,
\]
is equivalent to 
solving the equations
\begin{align*}
\nd_{a}\Phi^a=&1,\\
\nd_{[a}\Phi_{b]}=&0,\\
\nd_aX^a+X_aX^a-2 c_2 K-2\la=&0,
\end{align*}
where $\Phi_a=\frac{X_a}{X_bX^b}$.
\end{proposition} 
\begin{proposition}[Laplace]\label{Laplace}
The 2-dimensional generalised Ricci soliton equation with $c_1=0$
\[
\nd_{(a}X_{b)}-(c_2K+\la)g_{ab}=0
\]
is equivalent to the equations
\begin{equation}
\begin{aligned}
\nd_{a}\Phi^a=&0,\\
\nd_{[a}\Phi_{b]}=&0,\\
\nd_aX^a-2 c_2 K-2\la=&0,\label{zeji}
\end{aligned}
\end{equation}
where the relation between $X$ and $\Phi$ is given by $\Phi_a=\frac{X_a}{X_bX^b}$.
\end{proposition} 
The proofs of the propositions are straightforward. They parallel the proof of Theorem 4 in \cite{stable}, and are therefore omitted.

\section{2D examples of generalised Ricci solitons}\label{sec:2dex}

We now produce explicit examples of 2-dimensional generalised Ricci solitons $(g,X)$ of both signatures: Riemannian $(++)$ and Lorentzian $(+-)$. In the Lorentzian situation we only consider the case when $X_aX^a \neq 0$, i.e.\ when the soliton $(g,X)$ vector $X$ is non-null. The case where $X$ is null will be treated in Section \ref{null}.

In the Riemannian setting, we use the following metric ansatz:
\begin{equation}
g=A(y)\der x^2+\frac{1}{A(y)}\der y^2,\label{mansatz}\end{equation}

The Gaussian curvature of this metric is: 
\[
K=-\frac{1}{2}A''(y).
\]

A similar metric ansatz will be used in the Lorentzian setting. In this situation we will use the metric
\begin{equation}
g=A(y)\der x^2-\frac{1}{A(y)}\der y^2,\label{lool}
\end{equation}
where again $A(y)$ is a function of a single variable $y$. For this metric, 
\[
K=\frac{1}{2}A''(y).
\]

In the rest of this Section we will impose the generalised Ricci soliton equations (\ref{grs0}) on the pairs $(g,X)$, with $g$ being one of (\ref{mansatz}) or (\ref{lool}). We will first give the solutions with $c_1=0$, and then with $c_1\neq 0$.  Some important classical examples, such as the Hamilton cigar soliton and its Lorentzian counterpart, as well as the generalised Ricci soliton describing the extremal Kerr black hole horizon, will be obtained as special cases. 
\subsection{2D Ricci solitons and homotheties}
To get explicit examples of Riemannian $c_1=0$ generalised Ricci solitons $(g,X)$, we take the metric $g$ as in (\ref{mansatz}), and the 1-form $X$ given by
\begin{align}\label{laplacex}
X=A(y)\nu \der x+\mu\der y,
\end{align}
with  $\mu$, $\nu$ real constants, such that $\mu^2+\nu^2\neq 0$. We then have:
\begin{align*}
F=\der X=&-\nu A'(y)\der x \wedge \der y,
\end{align*}
so the gradient case, $F\equiv 0$, is obtained when either $\nu=0$ or $g$ is flat.

It turns out that for such $X$ the generalised Ricci soliton equations with $c_1=0$ reduce to a single second order ODE
\begin{align}\label{rssol}
c_2 A''+\mu A'-2 \la=0. 
\end{align}
This is the only equation to be solved for the ansatz (\ref{mansatz}) and (\ref{laplacex}) to obtain $c_1=0$ Ricci solitons.

\begin{remark}
If we refer back to Section \ref{potential}, then the 1-form $\Phi_a$ related to our $X_a$ via $\Phi_a=\frac{X_a}{X_bX^b}$, identically satisfies the Laplace condition given by the first two equations of Proposition \ref{Laplace}.  The remaining part of the system (\ref{zeji}), namely the scalar equation for the trace $\nd^aX_a$, becomes our ODE (\ref{rssol}).
\end{remark} 

The most general solution to (\ref{rssol}) when $c_2 \neq 0$ and $\mu\neq 0$ is
$$
A \left(y\right)=2\frac{\la}{\mu}y+\alpha {\rm e}^{-\frac{\mu}{c_2}y}+\beta,
$$
where $\alpha$, $\beta$ are constants. 

If $\mu=0$, the general solution is:
\begin{align}\label{1c10sol}
A \left(y\right)=\frac{\la}{c_2}y^2+\alpha y+\beta,
\end{align}
but the solution has constant Gaussian curvature $K=-\tfrac{\la}{c_2}$. Likewise, if $c_2=0$ and $\mu\neq 0$ we obtain the flat metric with the general solution to (\ref{rssol}) given by: 
\begin{align}\label{2c10sol}
A \left(y\right)=2\frac{\la}{\mu}y+\beta.
\end{align}
Eventually, the most degenerate case: $c_2=0$ and $\mu=0$, gives $\la=0$, with any arbitrary $A(y)$ being a solution to (\ref{rssol}). This last fact says that the vector field $X=\partial_x$, corresponding to the 1-form $X=A(y)\nu \der x$ is always a Killing symmetry for the metric (\ref{mansatz}), whatever $A=A(y)$ is. 

Now restricting to the general non-constant-curvature and non-Killing case we have the following proposition: 
\begin{proposition}\label{hcc}
For every $\lambda$ and $c_2\neq 0$ there is a 4-parameter family of Riemannian Ricci solitons, parametrised by $(\alpha,\beta,\mu\neq 0,\nu)$ explicitly given by:
$$
\boxed{\begin{aligned}
g=&
\Big(2\frac{\la}{\mu}y+\alpha {\rm e}^{-\frac{\mu}{c_2}y}+\beta\Big)\der x^2+\frac{1}{2\frac{\la}{\mu}y+\alpha {\rm e}^{-\frac{\mu}{c_2}y}+\beta}\der y^2
,\\ 
X=&\nu\Big(2\frac{\la}{\mu}y+\alpha {\rm e}^{-\frac{\mu}{c_2}y}+\beta\Big)\der x+\mu  \der y.
\end{aligned}
}
$$
The Gauss curvature of the metric is:
$$K=-\tfrac{\al\mu^2}{2c_2^2}{\rm e}^{-\frac{\mu}{c_2}y},$$
and the 2-form $F_{ab}$ is given by
$$F~=~\nu~\frac{\alpha\mu^2{\rm e}^{-\frac{\mu}{c_2}y}-2c_2\la}{\mu c_2}~\der x\wedge \der y.$$
\end{proposition}

For the Lorentzian metric (\ref{lool}), we take the same $X$ given by (\ref{laplacex}) but now with real constants  $\mu$, $\nu$ such that $\nu^2-\mu^2\neq 0$. The generalised Ricci soliton equations with $c_1=0$ in the Lorentzian case reduce to a single second order ODE
\begin{align}\label{rssollor}
c_2 A''+\mu A'+2 \la=0. 
\end{align}

The most general solution to (\ref{rssollor}) when $c_2 \neq 0$ and $\mu\neq 0$ is
$$
A \left(y\right)=-2\frac{\la}{\mu}y+\alpha {\rm e}^{-\frac{\mu}{c_2}y}+\beta,
$$
where $\alpha$, $\beta$ are constants; the other cases besides the most degenerate one give metrics of constant curvature.

Specialising to the case $c_2=-1$, we have

\begin{proposition}
There is a 4-parameter family of Lorentzian Ricci solitons explicitly given by:
$$
\boxed{\begin{aligned}
g=&
\left(-2\frac{\la}{\mu}y+\alpha {\rm e}^{\mu y}+\beta \right)\der x^2-\frac{1}{-2\frac{\la}{\mu}y+\alpha {\rm e}^{\mu y}+\beta}\der y^2
,\\ 
X=&\nu\Big(-2\frac{\la}{\mu}y+\alpha {\rm e}^{\mu y}+\beta\Big)\der x+\mu  \der y.
\end{aligned}
}
$$
\end{proposition}

\begin{example}[Generalised Hamilton's cigar] Let us now specialise to the steady case, $\la=0$.

In such case the formulae from Proposition \ref{hcc} simplify significantly if we introduce a variable $r$ related to $y$ via:
$$y=\frac{c_2}{\mu}\log\frac{\alpha}{\beta\big(\tanh^2(\tfrac{\mu\sqrt{\beta}}{2c_2}r)-1\big)}.$$
With such $r$, and with $\lambda =0$, the metric becomes:
\begin{equation}g=\der r^2+\beta \tanh^2(\tfrac{\mu\sqrt{\beta}}{2c_2}r)\der x^2,\label{hcs9}\end{equation} and the soliton $X$ becomes:
\begin{equation}X=\beta\nu \tanh^2(\tfrac{\mu\sqrt{\beta}}{2c_2}r)\der x+\mu \sqrt{\beta} \tanh(\tfrac{\mu\sqrt{\beta}}{2c_2}r)\der r.\label{hcs8}\end{equation}
As 
$$F=-\frac{\beta^{\frac{3}{2}}\mu\nu}{c_2}{\rm sech}^2(\tfrac{\mu\sqrt{\beta}}{2c_2}r) \tanh(\tfrac{\mu\sqrt{\beta}}{2c_2}r)\der x\wedge\der r,$$
and since we have assumed that $\mu\neq 0$, the soliton is a gradient if and only if $\be\nu=0$. We also note that the Gauss curvature is now:
$$K=\frac{\beta\mu^2}{2c_2^2}{\rm sech}^2(\tfrac{\mu\sqrt{\beta}}{2c_2}r),$$
so $\beta\neq 0$ is needed for the metric not to be flat. 

Now it is worthwhile to note that the metric (\ref{hcs9}) and the vector field (\ref{hcs8}) become the Hamilton cigar soliton if $\nu=0$ and $\beta>0$. The solution (\ref{hcs9})-(\ref{hcs8}) for $\nu\neq 0$ and $\beta>0$ gives a 1-parameter of nongradient Ricci solitons, on the background of the Hamilton cigar metric (\ref{hcs9}). They degenerate to the classical Hamilton gradient Ricci soliton when $\nu\to 0$. 

Note also that the solution  (\ref{hcs9})-(\ref{hcs8}) for the steady Ricci soliton equations makes sense for $\beta<0$. Despite of the appearance of $\sqrt{\beta}$ in formulas for $g$ and $X$, these formulas are real and give rise to a Riemannian metric $g$ even for $\beta<0$. Rewriting $\beta<0$ as $-|\beta|$ we get the following another real form of the generalised Hamilton's cigar:
 $$g=\der r^2+|\beta| \tan^2(\tfrac{\mu\sqrt{|\beta|}}{2c_2}r)\der x^2,$$
$$X=|\beta|\nu \tan^2(\tfrac{\mu\sqrt{|\beta|}}{2c_2}r)\der x-\mu\sqrt{|\beta|} \tan(\tfrac{\mu\sqrt{|\beta|}}{2c_2}r)\der r.$$
The 1-form $X$  has
$$F=-\frac{|\beta|^{\frac{3}{2}}\mu\nu}{c_2}{\rm sec}^2(\tfrac{\mu\sqrt{|\beta|}}{2c_2}r)\tan(\tfrac{\mu\sqrt{|\beta|}}{2c_2}r)\der x\wedge\der r,$$
and the metric $g$ has the Gauss curvature:
$$K=-\frac{|\beta|\mu^2}{c_2^2}{\rm sec}^2(\tfrac{\mu\sqrt{|\beta|}}{2c_2}r),$$
so again it represents a non-trivial gradient soliton if $\nu=0$. 
In \cite{ricsol2}, it is called the exploding soliton and the metric is incomplete.

Summarising we have the following 
\begin{proposition}\label{cigar}
For steady case ($\la=0$), the 4-parameter family of Riemannian Ricci solitons $(c_1=0, c_2=-1)$ obtained in Proposition \ref{hcc}, specialise to either the complete Hamilton cigar given by
$$
\boxed{\begin{aligned}
g=&\der r^2+\beta \tanh^2(\tfrac{\mu\sqrt{\beta}}{2}r)\der x^2,\\
X=&\beta\nu \tanh^2(\tfrac{\mu\sqrt{\beta}}{2}r)\der x-\mu \sqrt{\beta} \tanh(\tfrac{\mu\sqrt{\beta}}{2}r)\der r,
\end{aligned}
}
$$
for $\beta >0$, 
or the incomplete `exploding' soliton given by
$$
\boxed{\begin{aligned}
g=&\der r^2+|\beta| \tan^2(\tfrac{\mu\sqrt{|\beta|}}{2}r)\der x^2,\\
X=&|\beta|\nu \tan^2(\tfrac{\mu\sqrt{|\beta|}}{2}r)\der x+\mu\sqrt{|\beta|} \tan(\tfrac{\mu\sqrt{|\beta|}}{2}r)\der r.
\end{aligned}
}
$$
for $\beta <0$.
 \end{proposition}
\end{example}

\subsection{Examples with nonvanishing quadratic term}
\subsubsection{An ansatz for the Riemannian case}
If the quadratic in $X$ term in (\ref{grs0}) does not vanish we may always put $c_1=1$. For the $\boxed{c_1=1}$ generalised Ricci soliton equations, we take the Riemannian metric ansatz (\ref{mansatz}) and the 1-form
\begin{equation}
X=A(y)p(y) \der x+q(x,y)\der y,\label{ansc1}
\end{equation} 
with $q=q(x,y)$ being a function of both variables $x$ and $y$, and $p=p(y)$ being a function of variable $y$, only. 

With this ansatz we solve the generalised Ricci soliton equation (\ref{grs0}) with $c_1=1$ in three steps. In step one we solve for $q=q(x,y)$ from the $\der x\der x$ component of this equation. This gives:
$$q(x,y)=\frac{2\lambda -2A(y)p(y)^2-c_2 A''(y)}{A'(y)}.$$
In particular, this means that the function $q(x,y)$ can not depend on $x$, and hence we have $q(x,y)=q(y)$. Now, in step two, inserting this function back in (\ref{ansc1}), and looking at the component $\der x\der y$ of (\ref{grs0}) for $X$ with such $q(x,y)$, enables us to solve for $A''(y)$. This can be only done when $$c_2 p(y)\neq 0.$$ In such case $A''(y)$ is given by:
$$A''(y)=\frac{4\lambda p(y)-4A(y)p(y)^3+A'(y)p'(y)}{2c_2 p(y)}.$$
In step three we use the information about $A''(y)$ and look at the last of equations (\ref{grs0}), the one for the component $\der y\der y$. This reduces to an ODE 
$$2p(y)p''(y)-3p'(y)^2+4p(y)^4=0$$
for $p(y)$ that can be solved easily. The general solution is:
$$p(y)=\frac{\gamma}{1+\gamma^2(y+\beta)^2},$$
where $\gamma\neq 0$ and $\beta$ are real constants. We now insert this $p(y)$ in the equation for $A''(y)$ obtaining an ODE for $A=A(y)$. This, when solved, gives us the general solution to the the generalised Ricci soliton equation with $c_1=1$ and $c_2\neq 0$. 

To make this general solution more transparent, it is convenient to pass to new variables $(x,y,\gamma)\to(x,z,\alpha)$, where $$z=\gamma(y+\beta),\quad\quad \gamma=\alpha^{-1}.$$
This brings the equation for $A=A(y)$, which now is considered as a function $A=A(z)$, into:
\begin{equation}c_2(1+z^2)^2A''(z)+z(1+z^2)A'(z)+2A(z)=2\la\alpha^2(1+z^2)^2.\label{legr}\end{equation}
Our aim now is to show that this, is closely related to the (associated) Legendre equation. 

Indeed, changing the independent variable $A=A(z)$ into a new variable $B=B(z)$, which relates to $A=A(z)$ via:
\begin{equation}A(z)=(1+z^2)^{\tfrac{2c_2-1}{4c_2}}B(z)+\frac{\la\alpha^2(1+z^2)}{1+c_2},\label{cvar}\end{equation}
we change (\ref{legr}) into the (associated) Legendre equation for $B=B(z)$ (see Remark \ref{legrem}). Of course, this is only possible when $\boxed{c_2\neq -1}$. In such case we arrive at the following proposition.
\begin{proposition}\label{ell}
If $c_2\neq 0,-1$ and $X\wedge\der z\neq 0$, the most general generalised $c_1=1$ Ricci soliton corresponding to the ansatz (\ref{mansatz}), (\ref{ansc1}) is given by:
$$\boxed{\begin{aligned}
&g=\Big((1+z^2)^{\tfrac{2c_2-1}{4c_2}}B(z)+\la\frac{\alpha^2(1+z^2)}{1+c_2}\Big)\der x^2+\frac{\alpha^2\der z^2}{(1+z^2)^{\tfrac{2c_2-1}{4c_2}}B(z)+\la\frac{\alpha^2(1+z^2)}{1+c_2}},\\
&X=\Big(\frac{(1+z^2)^{-\tfrac{1+2c_2}{4c_2}}}{\alpha}B(z)+\frac{\la\alpha}{1+c_2}\Big)\der x+\frac{z}{1+z^2}\der{z}.
\end{aligned}}$$
Here $\alpha\neq 0$ is a constant, and the function $B=B(z)$ satisfies an ODE:
\begin{equation}(1+z^2)B''(z)+2zB'(z)-\Big((\tfrac{1}{2c_2}-1)\tfrac{1}{2c_2}-\frac{(\tfrac{1}{2c_2}+1)^2}{1+z^2}\Big)B(z)=0.\label{legqe}\end{equation}
The Gauss curvature of the metric $K$ is:
$$K=\frac{(1+z^2)^{-\tfrac{1+2c_2}{4c_2}}}{2c_2\alpha^2}\Big(\big(1-\tfrac{1}{2c_2}+\frac{1+\tfrac{1}{2c_2}}{1+z^2}\big)B(z)+zB'(z)\Big)-\frac{\la}{1+c_2},$$
and the Maxwell form $F$ of the soliton is
$$F=\der X=\frac{(1+z^2)^{-\tfrac{1+6c_2}{4c_2}}}{\al}\Big((1+\tfrac{1}{2c_2})zB(z)-(1+z^2)B'(z)\Big)\der x\wedge\der z.$$
\end{proposition}
\begin{remark}\label{legrem}
As we already mentioned, the equation (\ref{legqe}), under the substitution
$$z=i \xi,\quad\quad \ell=\tfrac{1}{2c_2}-1,\quad\quad m=\tfrac{1}{2c_2}+1,$$
becomes:
$$\boxed{(1-\xi^2)B''(\xi)-2\xi B'(\xi)+\Big(\ell(\ell+1)-\frac{m^2}{1-\xi^2}\Big)B(\xi)=0,}$$
i.e. the (associated) Legendre equation in its standard form. The solutions to this equation are given as linear combinations of the associated Legendre functions $P^m_{\ell}(\xi)$ and $Q^m_{\ell}(\xi)$.
\end{remark}
We have the following corollary.
\begin{corollary}
Every generalised Ricci soliton with $c_1=1$ and $c_2\neq 0, -1$ corresponding to the ansatz (\ref{mansatz}), (\ref{ansc1}), and satisfying $X\wedge \der z\neq 0$,  is obtained by using Proposition \ref{ell}, with the function $B=B(z)$ given by:
\begin{equation}B(z)=\beta P^{\ell+2}_{\ell}(-iz)+\gamma Q^{\ell+2}_{\ell}(-iz),\label{exb}\end{equation}
where $P^m_{\ell}$ and $Q^m_{\ell}$, with $\ell=\tfrac{1}{2c_2}-1$, are the associated Legendre functions, and the complex constants $\beta$ and $\gamma$ are chosen in such a way that the expression (\ref{exb}) is real.  
\end{corollary}

The associated Legendre functions arise in the study of spherical harmonics. As we see below, for some specific values of the parameter $c_2$, including the important case of $c_2=\tfrac12$, the function $B=B(z)$ can be expressed in terms of the elementary functions.  For the case of the generalised Ricci soliton equations with $c_2=-\frac{1}{N}$ where $N$ is some integer (related to the dimension $n$), such as the EW equation, or the equation determining metric projective structures with skew-symmetric Ricci tensor representative in its projective class, we have $\ell$ and $m$ taking integer and half-integer values. Again in these cases the associated Legendre functions in (\ref{exb}) reduce to elementary functions (and in some cases polynomials).

In the singular case $\boxed{c_2=-1}$, not much is changing when passing from (\ref{legr}) to the (almost) Legendre equation (\ref{legqe}). Simply, instead of replacing $A(z)$ with $B(z)$ via (\ref{cvar}), we make the change:
$$A(z)=(1+z^2)^{\tfrac{2c_2-1}{4c_2}}B(z)-\la\alpha^2z\sqrt{1+z^2}~{\rm arcsinh}z,$$
or, which is the same,
\begin{equation}A(z)=(1+z^2)^{\tfrac34}~B(z)-\la\alpha^2z\sqrt{1+z^2}~{\rm arcsinh}z.\label{cvarm1}\end{equation}
This brings the equation (\ref{legr}) for $A(z)$ with $c_2=-1$ into the equation for $B(z)$, which is the (almost) Legendre equation (\ref{legqe}). This for $c_2=-1$ becomes the (almost) associated Legendre equation
$$(1+z^2)B''(z)+2zB'(z)-\Big((\tfrac32\times\tfrac12-\frac{\big(\tfrac12\big)^2}{1+z^2}\Big)B(z)=0,$$
or the exact associated Legendre equation 
\begin{equation}(1-\xi^2)B''(\xi)-2\xi B'(\xi)+\Big(\ell(\ell+1)-\frac{(\ell+2)^2}{1-\xi^2}\Big)B(\xi)=0,\label{3212}\end{equation}
with $\ell=-\tfrac32$ (and $m=-\tfrac12$) for the variable $z=i\xi$.

This time, the coincidences in particular values of $\ell=-\tfrac32$ and $m=\ell+2=\tfrac12$, make the general solution to the associated Legendre equation (\ref{3212}), to be expressible in terms of elementary functions. Actually we have that 
$$B(\xi)=\tilde{\beta} (\xi^2 -1)^{\tfrac14}+\tilde{\gamma} \frac{\xi}{(\xi^2-1)^{\tfrac14}},$$
with complex constants $\tilde{\beta}$ and $\tilde{\gamma}$, 
is the general solution to 
$$(1-\xi^2)B''(\xi)-2\xi B'(\xi)+\Big(\tfrac34-\frac{\tfrac14}{1-\xi^2}\Big)B(\xi)=0.$$
Returning to the real variable $z$, using (\ref{cvarm1}), we conclude that this implies that the most general real solution for $A(z)$ of (\ref{legr}) with $c_2=-1$ is:
$$
A(z)=\beta(1+z^2)+\gamma z \sqrt{1+z^2}-\la \al^2 z \sqrt{1+z^2} ~{\rm arcsinh}z,
$$
with real constants $\beta$ and $\gamma$,
from which we get

\begin{proposition}\label{singular1}
If $c_2= -1$, the most general generalised $c_1=1$ Ricci soliton corresponding to the ansatz (\ref{mansatz}), (\ref{ansc1}) is given by:
$$\boxed{\begin{aligned}
&g=\Big(\beta(1+z^2)+\gamma z \sqrt{1+z^2}-\la \al^2 z \sqrt{1+z^2}~{\rm arcsinh}z\Big)\der x^2
\\
&\hspace{50pt}+\frac{\alpha^2}{\beta(1+z^2)+\gamma z \sqrt{1+z^2}-\la \al^2 z \sqrt{1+z^2}~{\rm arcsinh}z}\der z^2,\\
&X=\Big(\frac{\beta}{\alpha}+\frac{\gamma z}{\al \sqrt{1+z^2}}-\frac{\la \al z }{\sqrt{1+z^2}}~{\rm arcsinh}z\Big)\der x+\frac{z}{1+z^2}\der{z},
\end{aligned}}$$
where $\alpha\neq 0$ is a constant.
In this case, the 2-form $F$ is
$$F=\der X=\frac{1}{\al(1+z^2)^{\tfrac32}}\Big(\la \al^2 z \sqrt{1+z^2}+\la \al^2~{\rm arcsinh}z-\gamma\Big)\der x\wedge\der z.$$
\end{proposition}

For $\boxed{c_2=0}$ the ODE (\ref{legr}) reduces to a first order ODE, whose general solutions are:
\[
A(z)=\be \frac{1+z^2}{z^2}+\la \al^2 (1+z^2)
\]
with an arbitrary constant $\beta$.
In this case we have 
$$\boxed{\begin{aligned}
&g=(\la \al^2+\tfrac{\be}{z^2})(1+z^2)\der x^2+\frac{\al^2\der z^2}{(\la \al^2+\tfrac{\be}{z^2})(1+z^2)},\\
&X=\frac{1}{\al}(\la \al^2+\tfrac{\be}{z^2})\der x+\frac{z\der z}{1+z^2}.
\end{aligned}}$$

The transformation to $B(z)$ is not possible. Note that this solution is not well-defined on the set $\{z=0\}$.

For $\boxed{p(y)=0}$, the solutions obtained are gradient-like.

In this case we find that the $\der x \der x$ component of the $c_1=1$ generalised Ricci soliton equations with metric given by (\ref{mansatz}) and 1-form given  by 
\[
X=q(y) \der y,
\]
gives 
\[
A''(y)=\frac{2\la-q(y) A'(y)}{c_2}. 
\]
From this, substituting this expression into the $\der y \der y$ component we get
\[
q(y)^2+q'(y)=0, 
\]
whose general solutions are
$q(y)=\frac{1}{y+\beta}$.

Redefining coordinates by introducing $z=y+\beta$, we conclude that in this case the metric can be written as: $
g=A(z)\der x^2+\frac{1}{A(z)}\der z^2$,
and the 1-form $X$ as $X=\frac{\der z}{z}$.
The only equation to be satisfied for this pair to be 
the $c_1=1$ generalised Ricci soliton is an ODE
\begin{equation}\label{c11ode}
c_2 z A''(z)+A'(z)-2\la z=0
\end{equation}
for the function $A=A(z)$.

General solution to (\ref{c11ode}) depends on whether or not $c_2(c_2^2-1)$ is zero. In any case the solution is always expressible in terms of elementary functions. The corresponding solitons have
$$\boxed{X=\frac{\der z}{z}}$$
and
$$\begin{aligned}&\boxed{g=\Big( \frac{\alpha c_2 z^{{\frac{c_2-1}{c_2}}}}{c_2-1}+\frac{\la z^{2}}{c_2+1}+\beta\Big)\der x^2+\frac{\der z^2}{ \frac{\alpha c_2 z^{{\frac{c_2-1}{c_2}}}}{c_2-1}+\frac{\la z^{2}}{c_2+1}+\beta}\quad{\rm if}\quad c_2(c_2^2- 1)\neq 0,}\\
&\boxed{g=\Big(\la z^{2}+\beta\Big)\der x^2+\frac{\der z^2}{\la z^{2}+\beta}\quad{\rm if}\quad c_2= 0,}\\
&\boxed{g=\Big(\al \ln z+\beta+\frac{\la}{2}z^2\Big)\der x^2+\frac{\der z^2}{\al \ln z+\beta+\frac{\la}{2}z^2}\quad{\rm if}\quad c_2=1,}\\
&\boxed{g=\Big(\frac{\al}{2}z^2 +\beta+\frac{\la}{2}z^2-\la z^2 \ln z\Big)\der x^2+\frac{\der z^2}{\frac{\al}{2}z^2 +\beta+\frac{\la}{2}z^2-\la z^2 \ln z}\quad{\rm if}\quad c_2=-1,}\end{aligned}$$
with $\al$ and $\beta$ being arbitrary constants.

\subsubsection{An ansatz for the Lorentzian non-null case}
For the Lorentzian ansatz 
\begin{eqnarray}
&g=A(y)\der x^2-\frac{1}{A(y)}\der y^2,\nonumber\\
&X=A(y) p(y) \der x+q(y)\der y,\label{xlor}
\end{eqnarray}
we proceed in the same way as we did in the Riemannian case. In particular we obtain an ODE for $p(y)$ as before. In the case $c_2 p(y) \neq 0$, the ODE is 
\[
4 p(y)^4-2 p(y) p''(y)+3 p(y)^2=0,
\]
and has 
\[
p(y)=\frac{\gamma}{\ga(y+\beta)^2-1} 
\]
as a general solution.

Making the change of variables $(x,y,\gamma) \mapsto (x,z,\alpha)$ via $z=\gamma(y+\beta)$, $\alpha=\ga^{-1}$ as before,  we reduce all the generalised Ricci soliton equations for our ansatz to a single ODE for $A(z)$ given by
\begin{equation}c_2(1-z^2)^2A''(z)-z(1-z^2)A'(z)-2A(z)=-2\la\alpha^2(1-z^2)^2.\label{leglor}\end{equation}

Restricting now to the case when $c_2\neq 0,-1$, we again 
change the independent variable $A=A(z)$ into a new variable $B=B(z)$, which relates to $A=A(z)$ via:
\begin{equation}A(z)=(1-z^2)^{\tfrac{2c_2-1}{4c_2}}B(z)+\frac{\la\alpha^2(1-z^2)}{1+c_2}.\label{cvaru}\end{equation}
This changes (\ref{leglor}) into the (associated) Legendre equation for $B=B(z)$. In such case we arrive at the following proposition.
\begin{proposition}\label{ellor}
If $c_2\neq 0,-1$ and $X\wedge\der z\neq 0$, the most general generalised $c_1=1$ Ricci soliton corresponding to the ansatz (\ref{lool}), (\ref{xlor}) is given by:
$$\boxed{\begin{aligned}
&g=\Big((1-z^2)^{\tfrac{2c_2-1}{4c_2}}B(z)+\la\frac{\alpha^2(1-z^2)}{1+c_2}\Big)\der x^2-\frac{\alpha^2\der z^2}{(1-z^2)^{\tfrac{2c_2-1}{4c_2}}B(z)+\la\frac{\alpha^2(1-z^2)}{1+c_2}},\\
&X=-\Big(\frac{(1-z^2)^{-\tfrac{1+2c_2}{4c_2}}}{\alpha}B(z)+\frac{\la\alpha}{1+c_2}\Big)\der x-\frac{z}{1-z^2}\der{z}.
\end{aligned}}$$
Here $\alpha\neq 0$ is a constant, and the function $B=B(z)$ satisfies the associated Legendre equation:
\begin{equation}(1-z^2)B''(z)-2zB'(z)+\Big(\ell(\ell+1)-\frac{(\ell+2)^2}{1-z^2}\Big)B(z)=0,\label{legqelo}\end{equation}
with $\ell=\tfrac{1}{2c_2}-1$.
The Gauss curvature of the metric $K$ is:
$$K=\frac{(1-z^2)^{-\tfrac{1+2c_2}{4c_2}}}{2c_2\alpha^2}\Big(\big(1-\tfrac{1}{2c_2}+\frac{1+\tfrac{1}{2c_2}}{1-z^2}\big)B(z)+zB'(z)\Big)-\frac{\la}{1+c_2},$$
and the Maxwell form $F$ of the soliton is
$$F=\der X=\frac{(1-z^2)^{-\tfrac{1+6c_2}{4c_2}}}{\al}\Big((1+\tfrac{1}{2c_2})zB(z)+(1-z^2)B'(z)\Big)\der x\wedge\der z.$$
\end{proposition}

\begin{remark}
Again if we refer back to Section \ref{potential}, the 1-form $\Phi_a$ related to the $X_a$ from Proposition \ref{ellor} via $\Phi_a=\frac{X_a}{X_bX^b}$, satisfies the Poisson condition given by first two equations of Proposition \ref{Poisson}. The ODE (\ref{leglor}) is the remaining trace equation in Proposition \ref{Poisson}.
\end{remark} 
We have the following:
\begin{corollary}
Every Lorentzian generalised Ricci soliton with $c_1=1$ and $c_2\neq 0, -1$ corresponding to the ansatz (\ref{lool}), (\ref{xlor}) is obtained by using Proposition \ref{ellor}, with the function $B=B(z)$ given by:
\begin{equation}B(z)=\beta P^{\ell+2}_{\ell}(z)+\gamma Q^{\ell+2}_{\ell}(z),\label{exc}\end{equation}
where $\ell=\tfrac{1}{2c_2}-1$ and the functions $P^m_{\ell}$ and $Q^m_{\ell}$ are the associated Legendre functions. Note that contrary to the Riemannian case, now the constants $\beta$ and $\gamma$ parametrising the solutions are real.  
\end{corollary} 
 
In the singular case $c_2=-1$, we change from $A(z)$ to $B(z)$ via:
$$A(z)=(1-z^2)^{\tfrac{2c_2-1}{4c_2}}B(z)+\la\alpha^2z\sqrt{1-z^2}~{\rm arcsin}z,$$
or, which is the same,
$$A(z)=(1-z^2)^{\tfrac34}~B(z)+\la\alpha^2z\sqrt{1-z^2}~{\rm arcsin}z.$$
This brings the equation (\ref{leglor}) for $A(z)$ with $c_2=-1$ into the equation for $B(z)$, which is the associate Legendre equation (\ref{legqelo}) with $\ell=-\tfrac32$. Using its general solution 
$$B(z)=\beta (1-z^2)^{\tfrac14}+\gamma \frac{z}{(1-z^2)^{\tfrac14}},$$
with real constants $\beta$ and $\gamma$, 
we get the most general solution for $A(z)$ of (\ref{leglor}) with $c_2=-1$:
$$
A(z)=\beta(1-z^2)+\gamma z \sqrt{1-z^2}+\la \al^2 z \sqrt{1-z^2} ~{\rm arcsin}z,
$$
from which we get
\begin{proposition}\label{singular2}
If $c_2= -1$, the most general generalised $c_1=1$ Ricci soliton corresponding to the ansatz (\ref{lool}), (\ref{xlor}) is given by:
$$\boxed{\begin{aligned}
&g=\Big(\beta(1-z^2)+\gamma z \sqrt{1-z^2}+\la \al^2 z \sqrt{1-z^2}~{\rm arcsin}z\Big)\der x^2
\\
&\hspace{50pt}-\frac{\alpha^2}{\beta(1-z^2)+\gamma z \sqrt{1-z^2}+\la \al^2 z \sqrt{1-z^2}~{\rm arcsin}z}\der z^2,\\
&X=-\Big(\frac{\beta}{\alpha}+\frac{\gamma z}{\al \sqrt{1-z^2}}+\frac{\la \al z }{\sqrt{1-z^2}}~{\rm arcsin}z\Big)\der x-\frac{z}{1-z^2}\der{z},
\end{aligned}}$$
where $\alpha\neq 0$ is a constant.
In this case, the 2-form $F$ is
$$F=\der X=\frac{1}{\al(1-z^2)^{\tfrac32}}\Big(\la \al^2 z \sqrt{1-z^2}+\la \al^2~{\rm arcsin}z+\gamma\Big)\der x\wedge\der z.$$
\end{proposition}

For $\boxed{c_2=0}$ the ODE (\ref{leglor}) reduces to a first order ODE, with general solution:
\[
A(z)=\be \frac{1-z^2}{z^2}+\la \al^2 (1-z^2)
\]
with an arbitrary constant $\beta$.
In this case we have 
$$\boxed{\begin{aligned}
&g=(\la \al^2+\tfrac{\be}{z^2})(1-z^2)\der x^2-\frac{\al^2\der z^2}{(\la \al^2+\tfrac{\be}{z^2})(1-z^2)},\\
&X=-\frac{1}{\al}(\la \al^2+\tfrac{\be}{z^2})\der x-\frac{z\der z}{1-z^2}.
\end{aligned}}$$

The transformation to $B(z)$ is not possible. Note that this solution is not well-defined on the set $\{z=0\}$.

For $\boxed{p(y)=0}$, the solutions obtained are gradient-like. 

Repeating what we have done in the $p(y)=0$ Riemannian case, we find that the Lorentzian $c_1=1$ generalised Ricci soliton equations are solved by the metric given by
\[
g=A(z)\der x^2-\frac{\der  z^2}{A(z)},
\] 
and 1-form $X$ given  by $X=\frac{\der z}{z}$,
provided that the function $A(z)$ satisfies an ODE
\begin{equation}\label{lorode}
c_2 z A''(z)+A'(z)+2\la z=0. 
\end{equation}

General solution to (\ref{lorode}) depends on whether or not $(c_2^2-1)c_2$ is zero. Like the Riemannian situation the solution is always expressible in terms of elementary functions. The corresponding solitons have
$$\boxed{X=\frac{\der z}{z}}$$
and
$$\begin{aligned}&\boxed{g=\Big( \frac{\alpha c_2 z^{{\frac{c_2-1}{c_2}}}}{c_2-1}-\frac{\la z^{2}}{c_2+1}+\beta\Big)\der x^2-\frac{\der z^2}{ \frac{\alpha c_2 z^{{\frac{c_2-1}{c_2}}}}{c_2-1}-\frac{\la z^{2}}{c_2+1}+\beta}\quad{\rm if}\quad c_2(c_2^2-1)\neq 0,}\\
&\boxed{g=\Big(-\la z^{2}+\beta\Big)\der x^2-\frac{\der z^2}{-\la z^{2}+\beta}\quad{\rm if}\quad c_2= 0,}\\
&\boxed{g=\Big(\al \ln z+\beta-\frac{\la}{2}z^2\Big)\der x^2-\frac{\der z^2}{\al \ln z+\beta-\frac{\la}{2}z^2}\quad{\rm if}\quad c_2=1,}\\
&\boxed{g=\Big(\frac{\al}{2}z^2 +\beta-\frac{\la}{2}z^2+\la z^2 \ln z\Big)\der x^2-\frac{\der z^2}{\frac{\al}{2}z^2 +\beta-\frac{\la}{2}z^2+\la z^2 \ln z}\quad{\rm if}\quad c_2=-1,}\end{aligned}$$
with $\al$ and $\beta$ being arbitrary constants.

\subsection{Examples of 2D metric projective structures with skew-symmetric Ricci tensor representative}
The generalised Ricci soliton equations with parameters $c_1=1$, $c_2=-1$, $\la=0$, is the equation determining whether the projective class of the Levi-Civita connection of a given 2D metric admits skew-symmetric Ricci tensor representative (see \cite{skewricci2D}). In \cite{thesis} and \cite{skewricci2D} it is known as the projective Einstein-Weyl (pEW) equations. Projective structures with skew-symmetric Ricci tensor are of geometric interest for their relationship with 3-webs and Veronese webs (see \cite{veroneseweb}). 

In the Riemannian setting, solutions to the pEW equations are obtained by setting the parameters $c_1=1,$ $c_2=-1$, $\la=0$ in Proposition \ref{singular1}, while in the Lorentzian setting, solutions to the pEW equations are obtained by setting the parameters $c_1=1,$ $c_2=-1$, $\la=0$ in Proposition \ref{singular2}. 
The skew-symmetric part of the Ricci tensor is some constant multiple of $F$ (see \cite{skewricci2D}) for details. 
We obtain 
\begin{proposition}\label{pewe}
There is a 3-parameter family of Riemannian generalised Ricci soliton pairs satisfying the projective Einstein-Weyl (pEW) equation. They are given by:
$$\boxed{\begin{aligned}
&g=\Big(\beta(z^2+1)+\gamma z \sqrt{z^2+1}\Big)\der x^2+\frac{\alpha^2}{\beta(z^2+1)+\gamma z \sqrt{z^2+1}}\der z^2,\\
&X=\Big(\frac{\beta}{\alpha}+\frac{\gamma z}{\al \sqrt{z^2+1}}\Big)\der x+\frac{z}{1+z^2}\der{z},
\end{aligned}}$$
where $\alpha\neq 0$ is a constant.
In this case, the 2-form $F$ is
$$F=\der X=-\frac{\gamma}{\al(z^2+1)^{\tfrac{3}{2}}}\der x\wedge\der z.$$
A similar example is obtained in the Lorentzian setting.  There is a 3-parameter family of Lorentzian generalised Ricci soliton pairs satisfying the pEW equations explicitly given by:
$$\boxed{\begin{aligned}
&g=\Big(\beta(1-z^2)+\gamma z \sqrt{1-z^2}\Big)\der x^2-\frac{\alpha^2}{\beta(1-z^2)+\gamma z \sqrt{1-z^2}}\der z^2,\\
&X=-\Big(\frac{\beta}{\alpha}+\frac{\gamma z}{\al \sqrt{1-z^2}}\Big)\der x-\frac{z}{1-z^2}\der{z},
\end{aligned}}$$
where $\alpha\neq 0$ is a constant.
In this case, the skew-symmetric Ricci tensor is given by some constant multiple of the 2-form
$$F=\der X=\frac{\gamma}{\al(1-z^2)^{\tfrac{3}{2}}}\der x\wedge\der z.$$
Observe that in order for $F$ to be non-vanishing we require $\ga$ to be non-zero. If $\ga=0$, we obtain a metric of constant curvature. This agrees with the fact that projectively Einstein or Ricci-flat surfaces in $2$ dimensions are projectively flat.  
\end{proposition}

For example, in the Riemannian case  taking $\alpha=1$, $\beta=1$, $\ga=1$ we have
\[
g=\Big(z^2+1+z \sqrt{z^2+1}\Big)\der x^2+\frac{1}{z^2+1+z \sqrt{z^2+1}}\der z^2
\] 
and $A(z)=z^2+1+z \sqrt{z^2+1}>0$ on $\R^2$.
The Gauss curvature for this example is given by
\begin{align*}
K=&-1-\frac{z(2z^2+3)}{2(z^2+1)^{\frac{3}{2}}}.
\end{align*}
The Liouville or Cotton tensor for metric projective structures is given by 
\begin{align*}
Y_a=\epsilon^{bc}Y_{bca}=\epsilon^{bc}(\nd_bR_{ca}-\nd_cR_{ba})
=&2\epsilon^b{}_{a}\nd_bK.
\end{align*}
For this example, 
\begin{align*}
Y=\frac{3(2z^2+1+2z\sqrt{z^2+1})}{(z^2+1+z\sqrt{z^2+1})(z^2+1)^{\frac{3}{2}}}\der x.
\end{align*}
The local obstructions obtained in \cite{skewricci2D} vanish for this example.
For the general solution with $\al=1$, we have the second order ODE associated to the projective structure given by 
\begin{align*}
{\frac{d^{2}z}{d{x}^{2}}}=&\frac{3}{2}\,{\frac { \left( 2\,\beta
z\sqrt {{z}^{2}+1}+\gamma(2{z}^{2}+1) \right)}{\sqrt {{z}^{2}+1} \left( \ga z\sqrt {{z}
^{2}+1}+\beta({z}^{2}+1) \right) }} \left( {\frac{dz}{dx}} \right) ^{2}\\
&+\,{\frac { \left( \ga z\sqrt {{z}^{2}+1}
+\beta({z}^{2}+1) \right)  \left( 2\beta z\sqrt {{z}^{2}+1}+\ga (2{z}^{2}+1)
 \right) }{2\sqrt {{z}^{2}+1}}}.
\end{align*}
\begin{remark}
If we take $\al=1, \beta=0, \ga=1$, so that $A(z)=z\sqrt{z^2+1}$ on $\R^2$ we observe there is an apparent singularity at the line $\{z=0\}$, but since the Gauss curvature given by
\[
K=-\frac{z (2z^2+3)}{2(z^2+1)^{\frac{3}{2}}}
\] is defined everywhere, this singularity at $z=0$ arises from the coordinates we have chosen.
\end{remark}

\begin{remark}\label{3dew}
Let $(g_{\Si},X_{\Si})$ denote a soliton pair from Proposition \ref{pewe} satisfying the pEW equations on a 2D Riemannian or Lorentzian manifold $(\Si, g_{\Si})$. Taking the conformal class of the product metric $g_{\Si}+dt^2$ on $\Si \times \R$ gives us a Riemannian  or Lorentzian Einstein-Weyl structure with the Weyl connection given by pulling back $X_{\Si}$ to $\Si \times \R$.  Higher dimensional solutions to pEW   (resp. Einstein-Weyl) equations on the product manifold $\Si \times \R^{n-2}$ can be obtained by solving the relevant generalised Ricci soliton on $\Si$ with parameters $(c_1, c_2, \la)=(1, -\frac{1}{n-1}, 0)$ (resp.  $(c_1, c_2 \la)=(1, -\frac{1}{n-2}, 0)$ ) and taking the product metric with the flat one on $\R^{n-2}$. 
\end{remark}

\subsection{Vacuum near-horizon geometries examples}
To get vacuum near-horizon geometries examples we specialise to the case $c_1=1$, $c_2=\frac{1}{2}$, $\la=0$ in Proposition \ref{ell}. We obtain
\begin{equation}
A \left( z \right)=B(z)=\beta\frac{ 2z}{1+z^2}+\gamma\frac{ (1-z^2)}{1+z^2}.\label{sin}
\end{equation}
\begin{remark}
The solution (\ref{sin}) is the general solution to the (almost) Legendre differential equation (\ref{legqe}) with the parameter $c_2=\frac{1}{2}$. Writing this equation explicitly we have
\begin{equation}(1+z^2)B''(z)+2zB'(z)+\frac{4}{1+z^2}B(z)=0.\label{sinu}\end{equation}
The familiar form $B_1(z)=\tfrac{2z}{1+z^2}$, $B_2(z)=\tfrac{1-z^2}{1+z^2}$ of the fundamental solutions constituting (\ref{sin}), and the fact that $B^2_1(z)+B_2^2(z)=1$, suggests the introduction of a variable $\theta$ such that
\begin{equation}\cos\theta= \frac{1-z^2}{1+z^2}\quad{\rm and}\quad\sin\theta=\frac{2z}{1+z^2}.\label{sini}\end{equation}
The variable change $z\to\theta$, given by (\ref{sini}), is the inverse stereographic projection from the unit circle parametrised by $\theta\in [0,2\pi]$ to the real line parametrised by $z\in]-\infty,\infty[$. The fact that the solution for $B(z)$ can be rewritten in the new variable $\theta$ as $$B(\theta)=\be \cos\theta+\ga\sin\theta$$ means, that under the variable change $z\to\theta$, the (almost) Legendre equation (\ref{sinu}) magically becomes the harmonic oscillator equation
$$B''(\theta)=-B(\theta).$$  
Actually making a more general variable change
$$\cos(\omega \theta)= \frac{1-z^2}{1+z^2}\quad{\rm and}\quad\sin(\omega\theta)=\frac{2z}{1+z^2}$$
we bring the (almost) Legendre equation with $\ell=0$, $m=2$ to the most general harmonic oscillator equation $B''(\theta)=-\omega^2 B(\theta)$.  
\end{remark}

\vspace{1truecm}
Using the solution (\ref{sin}) we get the following specialisation of the Proposition \ref{ell}. 
\begin{proposition}\label{nhge}
There is a 3-parameter family of generalised Ricci soliton pairs satisfying the near-horizon geometry equation. They are given by:
$$\boxed{\begin{aligned}
&g=\left(\frac{2\beta z+\gamma (1-z^2)}{1+z^2}\right)\der x^2+\left(\frac{\al^2(1+z^2)}{2\beta z+\gamma (1-z^2)}\right)\der z^2
,\\ 
&X= \frac{(2\beta z+\ga (1-z^2))}{\al (1+z^2)^2}\der x+\frac{z}{1+z^2}\der z.&
\end{aligned}}$$
The Gauss curvature for these solitons is:
$$K=\frac{2\ga (1-3z^2)}{\al^2(1+z^2)^3}+\frac{2\be z(3-z^2)}{\al^2(1+z^2)^3},$$
and the Maxwell 2-form 
$$F~=~\Big(~\frac{2\ga z(3-z^2)}{\al(1+z^2)^3}+\frac{2\be(3z^2-1)}{\al(1+z^2)^3}~\Big)~\der x\wedge\der z.$$
\end{proposition}
It is well known that in this class of solitons the extremal Kerr horizon geometry is included (see e.g. \cite{Jezierski}, \cite{stable}, \cite{KL0}). In the next section we will pick up this soliton using simple geometric analysis arguments.

\subsubsection{How singular can an extremal horizon with Killing symmetry be?}\label{gil}
Since the metric $g$ of the soliton given in Proposition \ref{nhge} has a Killing symmetry $\partial_x$, we interpret $g$ as a metric of a surface of revolution, with the azimuthal coordinate $x$. 

Looking at the formula for $g$ in Proposition \ref{nhge} we see that the metric is regular for all values of the coordinates $(x,y)$ except the points $(x,z)$ for which $A(z)=0$, or explicitly, at the points $(x,z)$ satisfying
$$\gamma z^2-2\be z-\ga=0.$$ 
Thus there are at most two values of $z$, for which $g$ is singular. 

We have two cases: either $\boxed{\gamma\neq 0}$ and we have two singular $z$'s, namely
$$z_\mp=\tfrac{\be}{\ga}\mp\sqrt{1+\tfrac{\be^2}{\ga^2}},$$
or $\boxed{\gamma=0}$ and we have only one $z$, namely $z_0=0$. 

Although the Gauss curvature $K(z)$ of the metric is regular at points where $A(z)=0$:
$$K_\mp=K(z_\mp)=-\tfrac{\ga}{2\al^2}(1\pm\tfrac{\be}{\sqrt{\be^2+\ga^2}}),\quad K(0)=K_0=0,$$
a further analysis is needed to determine if the singularity of the metric at $z_\mp$, $z_0$ comes from a bad choice of coordinates or if it is essential.

We first analyse the case when $\boxed{\gamma\neq 0}$.

Our interpretation of $g$ as a metric of a surface of revolution, and the fact that the $\der x^2$ term in the metric vanishes at $z_\pm$, enables us to think about the singular points $(x,z_\pm)$ as two antipodal points on the symmetry axis of the metric. Whether
these two points are in finite metric distances from the regular points of the
surface, and if so, whether they are regular or singular points of the surface, is to be
determined.

One way of detecting an essential singularity at a suspected point consists in passing to `polar' coordinates $(x,y)$ centred at this point. For our suspected point $z_-$, such `polar' coordinates are given by the relation: $(x,z)=(x,y^2+z_-)$, with the `pole' at the singular point $(x,y=0)$. To see if the `pole' is smooth, or if it has an essential singularity we check for the `conic angle' at the `pole'. This is the number $2\pi -\Phi$, where 
$$\Phi=\lim_{\epsilon\to 0}\frac{c(\epsilon)}{s(\epsilon)},$$
with $s(\epsilon)$ being the radius of a small metric circle centred at the `pole', and with $c(\epsilon)$ being the circumference of this circle. 

Only if $\Phi$ equals to $2\pi$, the `pole' is a smooth point. 

If $\Phi$ does not equal to $2\pi$, but if it is still a well defined real number, we have a relatively simple `conic' singularity, with the conic angle $2\pi-\Phi$. 

If $\Phi$ is not a real number - more complicated singularity occurs at the `pole'.  

In our case of the pole at $z_-$ we first transform the soliton metric to the new coordinates $(x,y)$ obtaining $g=g_{xx}(y)\der x^2+g_{yy}(y)\der y^2$, and set the range of the coordinate $x$ on circles tangent to $\partial_x$ to be $x\in[0,\chi]$. Then, for small $\epsilon$, we find that we have:
$$s_-(\epsilon)=\int_0^\epsilon \sqrt{g_{yy}(y)}~\der y~=~\frac{2\al}{\ga}~\epsilon~\sqrt{\sqrt{\be^2+\ga^2}-\be}+{\mathcal O}(\epsilon^2)$$
and
$$c_-(\epsilon)=\int_0^\chi \sqrt{g_{xx}(\epsilon)}~\der x=\chi~\epsilon~\sqrt{\sqrt{\be^2+\ga^2}+\be}+{\mathcal O}(\epsilon^2).$$

Both numbers $s_-(\epsilon)$ and $c_-(\epsilon)$ are finite for small $\epsilon$. Because $\lim_{\epsilon\to 0}s_{-}(\epsilon)\neq 0$, also the angle $\Phi$ is a well defined real number:
$$\Phi_-~=~\frac{\chi}{2\al}~(\sqrt{\be^2+\ga^2}+\be).$$
Since $\Phi$ is well defined for all $\al\neq 0$,  we have at most `conic' singularity at $z_-$. 

We can remove this kind of singularity, very easily, by an appropriate choice of the upper limit $\chi$ of the interval $[0,\chi]$. For this we only need that 
$$\Phi_-=2\pi.$$
Solving this for $\chi$ gives
$$\chi=\chi_-=4\pi ~\frac{\al}{\ga}~(\sqrt{1+\tfrac{\be^2}{\ga^2}}-\tfrac{\be}{\ga}),$$
and the choice of the range for $x$ to be $x\in[0,\chi_-]$ makes $z_-$ a smooth point of the considered surface.

If we started with $z_+$ instead of $z_-$, we could change coordinates via $(x,z)=(x,-y^2+z_+)$, and perform a similar analysis as above, to show that $z_+$ is a `conic' singular point with 
$$\Phi_+=~=~\frac{\chi}{2\al}~(\sqrt{\be^2+\ga^2}-\be).$$
The conical singularity at this point could be smoothed out by choosing $\chi$ such that $\Phi_+=2\pi$, which would give the following upper limit for the azimuthal coordinate $x$:
$$\chi_+=4\pi ~\frac{\al}{\ga}~(\sqrt{1+\tfrac{\be^2}{\ga^2}}+\tfrac{\be}{\ga}).$$

Thus we can always interpret the metric of the soliton as a metric of a closed surface of revolution, smooth everywhere except one of the points $z_\pm$. If the non-smooth point is, say, at $z_\pm$, then its antipodal point at $z_\mp$ is smooth, provided that the azimuthal coordinate $x$ ranges from 0 to  $\chi_\mp$.   

To make this surface smooth also at the antipode of $z_\mp$, we have to identify originally unrelated `azimuthal angles' $x$ of the respective polar coordinate systems around $z_-$ and $z_+$. This in particular means that to have both points $z_-$ and $z_+$ smooth, we need to impose $$\chi_-=\chi_+.$$ 
This is possible if and only if $$\beta = 0.$$
Thus only solitons with $\beta=0$ may be interpreted as living on a smooth surface.

In the cases when $\beta\neq 0$ we still have solitons living on surfaces with topology of a 2-sphere, but in these cases the solitons surfaces always have one point with `conical singularity'. Smoothing out the point $z_\mp$, produces a conical singularity at the antipodal point $z_\pm$. A singularity with a definite `conical angle' equal to $2\pi\big(1-(\sqrt{1+\tfrac{\be^2}{\ga^2}}\mp\tfrac{\be}{\ga})\big)$.

Concluding this part of our analysis we note that although the horizon surfaces with $\be\neq 0$ have a singularity at one point, the singularity there is very mild. It is only `conical', as opposed to any kind of a sharper one. We have the following 
\begin{proposition}
If $\al\neq 0$, $\ga\neq 0$ and $\boxed{\be\neq0}$ we have two types of surfaces on which the 3-parameter family of near horizon geometries described by Proposition \ref{nhge} live. The surfaces in both types have topology of a 2-sphere, and they are smooth everywhere except one point, in which the surface has conical singularity. The conic angle at the singular point is $$2\pi\big(1-(\sqrt{1+\tfrac{\be^2}{\ga^2}}\mp\tfrac{\be}{\ga})\big),$$
where the $\mp$ sign distinguishes the two types.\end{proposition}

If $\al\neq 0$, $\ga\neq 0$ and $\boxed{\be=0}$, the surface of the near-horizon geometry from Proposition \ref{nhge} is a surface with topology of a 2-sphere, which is smooth everywhere. Since the surface is a smooth 2-sphere and its metric has Killing vector $\partial_x$, the corresponding near horizon geometry must coincide with Kerr's extremal horizon geometry by theorems of H{\'a}j{\'i}{\v c}ek \cite{haj}, Lewandowski, Pawlowski \cite{LP03} and Jezierski \cite{Jezierski}.

To see this explicitly we use formulas in Proposition \ref{nhge} with $\be=0$, and redefine the coordinate $x$ by $\sqrt{\ga}x\to x$, and the constant $\al$ via $\frac{\al}{\sqrt{\ga}}\to\al$. This removes the redundant constant $\gamma$ from the considered family of solutions. We have the following 
\begin{corollary}\label{exkerr}
There is a 1-parameter family $(\boxed{\al\neq 0},\be=0,\ga=1)$ of generalised Ricci solitons from Proposition \ref{nhge} defining a 1-parameter family of near-horizon geometries given by:
$$\boxed{\begin{aligned}
&g=\frac{1-z^2}{1+z^2}\der x^2+\frac{\al^2(1+z^2)}{1-z^2}\der z^2
,\\ 
&X= \frac{1-z^2}{\al (1+z^2)^2}\der x+\frac{z}{1+z^2}\der z.&
\end{aligned}}$$
The extremal horizon lives on a smooth surface of revolution with topology of a 2-sphere. The surface is parametrised by $(x,z)$, with the following ranges: $0\leq x\leq 4\al\pi$, $-1\leq z\leq 1$. 

This near horizon geometry coincides with the Kerr extremal horizon with mass $M=\al$. The passage to the standard Kerr coordinates is given by:  $z=\cos(\theta)$, $x=2\phi$.
\end{corollary} 

In the case of $\boxed{\gamma=0}$ the metric $g$ appearing in Proposition \ref{nhge} has only one singular point at $z=0$. Using the arguments presented for the $\gamma\neq 0$ case, we show that this point can be interpreted as a smooth point on a surface $\Sigma$ parametrised by $(x,z)$, with the variable $x$ ranging from $0$ to $$\chi_o~=~2\pi~\frac{\al}{\be}.$$ This regularises the only singular point on $\Sigma$ and defines a smooth surface, with a near-horizon geometry structure on it. However, contrary to the case $\gamma\neq 0$, the surface $\Sigma$ is open, as the variable $z$ can now run from $z=0$ to $z=+\infty$, and because the length integral $$l=\int_0^{+\infty}\frac{\al}{\sqrt{\be}}\sqrt{\frac{1+z^2}{z}}\der z~>~\int_0^{+\infty}\frac{\al}{\sqrt{\be}}\sqrt{z}\der z$$ of any path from the `pole' $z=0$ to $z=\infty$, along constant $x$, diverges. We again redefine the coordinate $x$ via $\sqrt{\be}x\to x$, and the constant $\al$ via $\tfrac{\al}{\sqrt{\be}}\to\al$,
  to obtain the following 
\begin{proposition}
There is a 1-parameter family $(\boxed{\al\neq 0},\be=1,\ga=0)$ of generalised Ricci solitons from Proposition \ref{nhge}, defining a 1-parameter family of near-horizon geometries given by:
$$\boxed{\begin{aligned}
&g=\frac{2 z}{1+z^2}\der x^2+\frac{\al^2(1+z^2)}{2 z}\der z^2
,\\ 
&X= \frac{2 z}{\al (1+z^2)^2}\der x+\frac{z}{1+z^2}\der z.&
\end{aligned}}$$
The extremal horizon lives on a smooth open surface of revolution parametrised by $(x,z)$ with ranges: $0\leq x\leq 2\pi\al$, $0\leq z\leq +\infty$. 
\end{proposition}

For completeness we also present an example in the Lorentzian case. From Proposition \ref{ellor} we obtain
\begin{proposition}
There is a 3-parameter family of Lorentzian generalised Ricci soliton pairs satisfying the near-horizon geometry equation. They are given by:
$$
\boxed{\begin{aligned}
g=& \left(\frac{2\beta z+\gamma (1+z^2)}{1-z^2}\right)\der x^2-\left(\frac{\al^2(1-z^2)}{2\beta z+\gamma (1+z^2)}\right)\der z^2, \\ 
X=& -\frac{1}{\al (1-z^2)^2}\left( 2 \beta z+\ga (1+z^2)\right)\der  x-\frac{z}{1-z^2} \der z 
\end{aligned}
}
$$
\end{proposition}

\section{2D examples with null 1-form}\label{null}
We now pass to the 2-dimensional Lorentzian generalised Ricci soliton $(g,X)$, with $X_a$ being null. This case is not covered by our generalisation of Jezierski's approach in Section \ref{potential}, since we cannot rescale $X_a$ to get $\Phi_a$. We first need the following:
\begin{lemma}
Any 2-dimensional Lorentzian metric can be put into the form
\begin{equation}\label{nullg}
g=2 \der x \der y+H(x,y) dx^2.
\end{equation}
Any 1-form that is null with respect to this metric is of the form 
\begin{align*}
X=L(x,y) dx \hspace{12pt}\mbox{or} \hspace{12pt} X=\frac{1}{2}E(x,y)H(x,y) dx+E(x,y) dy
\end{align*}
for some functions $L(x,y)$ and $E(x,y)$.
\end{lemma}
\begin{proof}
Any 2-dimensional Lorentzian metric can be put into the form 
\[
g=2e^{2f(x,y)}\der x \der y
\]
for some function $f$. 
By introducing new coordinate $Y=y+h(x,y)$ for some function $h$ to be determined, we find that
\[
\der Y=\der y+h_x \der x+h_y \der y=h_x \der x+(1+h_y) \der y, 
\] 
upon which a substitution yields
\[
g=\frac{2}{1+h_y}e^{2f} \der x(\der Y-h_x \der x).
\]
Solving for $h$ in  
\[
1+h_y=e^{2f}
\]
allows us to put the metric into the form
\[
g=2\der x \der Y+H(x,Y) \der x^2
\]
for some function $H(x,Y)$ and we can redefine coordinates. 
For any vector field $X$ in these coordinates, we have
\[
X=L(x,y) \der x+E(x,y) \der y
\]
for some functions $L(x,y)$ and $E(x,y)$. The condition for $X$ to be null then implies
\[
E(x,y)(2 L(x,y)-H(x,y) E(x,y))=0, 
\]
so that either $E(x,y)=0$ or $L(x,y)=\frac{1}{2}H(x,y) E(x,y)$.
\end{proof}

For the metric ansatz given by (\ref{nullg})
\begin{align*}
g=H(x,y) \der x^2+2 \der x \der y, 
\end{align*}
we plug in the null 1-form given by
\begin{align*}
X=L(x,y)\der x 
\end{align*}
into the $c_1=0$ and $c_1=1$ generalised Ricci soliton equations. 
We find that $\der x\der y$ component of the equations in both cases determines the same $L(x,y)$:
\begin{align*}
0=-\la-\frac{c_2}{2}\,{\frac {\partial ^{2}}{\partial {
y}^{2}}}H \left( x,y \right)+\frac{1}{2}\,{\frac {\partial }{\partial y}}L
 \left( x,y \right),
\end{align*}
which implies
\begin{align*}
L \left( x,y \right) =c_2\,{\frac {\partial }{\partial y}}H
 \left( x,y \right) +2 \la y+f(x).
\end{align*}
A further computation shows that
\begin{align*}
F=\der X=\left(-c_2 \frac{\partial^2}{\partial y^2}H(x,y) -2 \la\right) \der x \wedge \der y.
\end{align*}

\subsection{2D Ricci solitons and homotheties}
For the metric ansatz (\ref{nullg}) and the null 1-form $X$ given by 
\begin{align}\label{nullX1}
X=(c_2\,{\frac {\partial }{\partial y}}H
 \left( x,y \right) +2 \la y+f(x)) \der x,
\end{align}
the $c_1=0$ generalised Ricci soliton equations (Ricci solitons and homotheties), reduce to a single non-linear second order PDE given by the
$\der x\der x$ component:
\begin{align}\label{c10pde}
0=&-\la H(x,y) -\frac{c_2}{2}H(x,y) {\frac {
\partial ^{2}}{\partial {y}^{2}}}H(x,y)+f'(x)+\frac{f(x)}{2}\frac{\partial}{\partial y}H(x,y)\nonumber\\
 &+\frac{1}{2}\, \left( {\frac {\partial 
}{\partial y}}H(x,y)\right) ^{2}c_2+ \left( {
\frac {\partial }{\partial y}}H(x,y)  \right) \la y+ c_2{\frac {\partial ^{2}}{\partial y\partial x}}H(x,y).
\end{align}
We want to solve this PDE. There are two cases: $c_2=0$ and $c_2\neq 0$.

If $c_2= 0$ and $\lambda\neq 0$, the generalised Ricci soliton equations are the equations for homotheties and (\ref{c10pde}) can be totally solved to obtain
\[
X= (2\la y+f(x))\der x, \qquad
H(x,y)=(2 \la y+f(x))h(x)+\frac{f'(x)}{\la}.
\]
Here $h(x)$ is an arbitrary function. The Ricci scalar for this metric is $0$, so the metric is flat. We also obtain the flat metric in the case when $c_2=0$ and $\lambda=0$. In this case $X=f(x)\der x$ and $H(x,y)=h(x)-2y (\log f(x))'$. 

If $c_2\neq 0$, the equation (\ref{c10pde}) is nonlinear in $H(x,y)$, and we can solve it only in special cases. For example, setting both $\la$ and $f(x)$ to be zero gives
\begin{align}\label{c10ricsol}
0=&-H(x,y) {\frac {
\partial ^{2}}{\partial {y}^{2}}}H(x,y)
 +\left( {\frac {\partial 
}{\partial y}}H(x,y)\right)^{2}+2{\frac {\partial ^{2}}{\partial y\partial x}}H(x,y).
\end{align}
This equation admits a solution in the form $$H(x,y)=A(x)B(y).$$ This is given by
\begin{align*}
A(x)=&\frac{1}{b-ax},\\
B(y)=&\frac{{{\rm e}^{c (y+d)}
}+2 a}{c},
\end{align*}
where $a$, $b$, $c$, $d$ are constants.
The Lorentzian metric given by this solution $H(x,y)=\frac{{\rm e}^{c (y+d)}
+2a}{c(b-ax)}$ admits a non-null Killing symmetry given by $\partial_x+\frac{a}{c(ax-b)}\partial_y$. 
We have
\begin{proposition}
There is a 4-parameter family of Lorentzian generalised Ricci soliton pairs satisfying the steady ($\la=0$) Ricci soliton equations given by:
$$
\boxed{\begin{aligned}
g=& 2 \der x \der y+\frac{{\rm e}^{c (y+d)}
+2a}{c(b-ax)} \der x^2
,\\ 
X=&-\frac{{\rm e}^{c(y+d)}}{b-ax} \der x.
\end{aligned}
}
$$
\end{proposition}

In the other case that $X$ is null for the metric (\ref{nullg}), the 1-form $X$ is given by
\[
X=\frac{1}{2}E(x,y)H(x,y) \der x+E(x,y) \der y.
\]
Plugging this ansatz for $X$ into the $c_1=0$ generalised Ricci soliton equations, we find that the $\der y \der y$ component gives 
\[
E_y=0, 
\]
from which we obtain
\[
E(x,y)=f(x).
\]
The remaining equations to solve is the following PDE:
\begin{align*}
2 c_2 H_{yy} (x,y)+4 \la+H_{y}(x,y) f(x)-2 f'(x)=0.
\end{align*}
Outside the singular locus defined by $f(x)=0$, its general solution for $c_2 \neq 0$ is given by
\begin{align*}
H(x,y)=&-\frac{2 c_2 {\rm e}^{-\frac{f(x) y}{2 c_2}} h(x)}{f(x)}+2\frac{(f'(x)-2\la)y}{f(x)}+j(x),
\end{align*}
where $h(x)$ and $j(x)$ are arbitrary functions. 
For $c_2=0$, we have
\begin{align*}
H(x,y)=&2\frac{(f'(x)-2\la)y}{f(x)}+j(x),
\end{align*}
and the metric with this $H(x,y)$ is flat. 
We obtain solutions to the $c_1=0$, $c_2 \neq 0$ generalised Ricci soliton equations given by
$$
\boxed{\begin{aligned}
g=& 2 \der x \der y+\left(-\frac{2 c_2 {\rm e}^{-\frac{f(x) y}{2 c_2}} h(x)}{f(x)}+2\frac{(f'(x)-2\la)y}{f(x)}+j(x)\right) \der x^2
,\\ 
X=&\frac{1}{2}\left(-2 c_2 {\rm e}^{-\frac{f(x) y}{2 c_2}} h(x)+2(f'(x)-2\la)y+j(x)f(x)\right) \der x+f(x) \der y.
\end{aligned}
}
$$
For this class of examples, we find that
\[
F=\der X=\left(-\frac{f(x) {\rm e}^{-\frac{f(x) y}{2 c_2}}h(x)}{2}-2\la\right) \der x \wedge \der y
\]
and generically $I_1$ and $I_2$ do not vanish, so that in general $g$ has no local Killing symmetry.

\subsection{2D examples with nonvanishing quadratic term}\label{nullc11}
For the same metric ansatz (\ref{nullg}) and null 1-form (\ref{nullX1}),  the $c_1=1$ generalised Ricci soliton equations again reduce to a single non-linear second order PDE on the functions $H(x,y)$ and $f(x)$, given by the $\der x\der x$ component. The PDE is
\begin{align}\label{longpde}
&-\la H(x,y) -\frac{c_2}{2}H(x,y) {\frac {
\partial ^{2}}{\partial {y}^{2}}}H(x,y) + \left({\frac{\partial }{\partial y}}H(x,y)\right)^{2}c_2^2+4\, \left( {\frac {\partial }{\partial y}}H(x,y) 
 \right) c_2 \la y \nonumber\\
 &+4\,{\la}^{2}{y}^{2}+2f(x)c_2 \frac{\partial}{\partial y}H(x,y)
 +4 f(x) \la y+f(x)^2+\frac{f(x)}{2}\frac{\partial}{\partial y}H(x,y)+f'(x)\nonumber \\
 &+\frac{1}{2}\, \left( {\frac {\partial 
}{\partial y}}H(x,y)\right) ^{2}c_2+ \left( {
\frac {\partial }{\partial y}}H(x,y)  \right) \la y+ c_2{\frac {\partial ^{2}}{\partial y\partial x}}H(x,y)=0.
\end{align}
We aim to solve this PDE, and again we have to consider cases.

The first case is when $c_2=0$ and $\lambda\neq 0$. In this situation the PDE (\ref{longpde}) reduces to 
\begin{align*}
&-\la H(x,y)+4\,{\la}^{2}{y}^{2}
 +4 f(x) \la y+f(x)^2+ \left( {
\frac {\partial }{\partial y}}H(x,y)  \right) \la y\\
&
+\frac{f(x)}{2}\frac{\partial}{\partial y}H(x,y)+f'(x)=0.
\end{align*}
This can be solved to obtain
\begin{align*}
H \left( x,y \right) = \big( 2\la y+f(x)  \big) \big(h(x) -2\,y  \big)+\frac{f'(x)}{\la} , 
\end{align*}
with $h=h(x)$ being arbitrary function of $x$. The Lorentzian metric with this $H(x,y)$ has constant scalar curvature equal to $-8 \la$. If $c_2=0$ and $\la=0$, the general solution to (\ref{longpde}) is $H(x,y)=h(x)-2y (h(x)+(\log h(x))')$, but for such $H(x,y)$ the Lorentzian metric is flat. 

In general case, when $c_1=1$ and $c_2\neq 0$, we simplify the equation (\ref{longpde}) by restricting to situations when the integration factor $f(x)\equiv 0$. In such cases (\ref{longpde}) reduces to a PDE on $H(x,y)$, which looks like that:
\begin{align}\label{nlpde}
&-\la H(x,y) -\frac{c_2}{2}H(x,y) {\frac {
\partial ^{2}}{\partial {y}^{2}}}H(x,y) + \left({\frac{\partial }{\partial y}}H(x,y)\right)^{2}c_2^2+4\, \left( {\frac {\partial }{\partial y}}H(x,y) 
 \right) c_2 \la y\nonumber \\
 &+4\,{\la}^{2}{y}^{2}+\frac{1}{2}\, \left( {\frac {\partial 
}{\partial y}}H(x,y)\right) ^{2}c_2+ \left( {
\frac {\partial }{\partial y}}H(x,y)  \right) \la y+ c_2{\frac {\partial ^{2}}{\partial y\partial x}}H(x,y)=0.
\end{align}
Further setting $\la=0$, and using the fact that now $c_2\neq 0$, we get
\begin{align}\label{nlpde2}
-H(x,y) {\frac {
\partial ^{2}}{\partial {y}^{2}}}H(x,y) + 2\left({\frac{\partial }{\partial y}}H(x,y)\right)^{2}c_2+\left( {\frac {\partial 
}{\partial y}}H(x,y)\right)^{2}+2{\frac {\partial ^{2}}{\partial y\partial x}}H(x,y) =0.
\end{align}
While we do not know what is the general solution to the PDE (\ref{nlpde2}), we can find its particular solutions by separation of variables with $H(x,y)=A(x)B(y)$. For such an ansatz the equation (\ref{nlpde2}) reduces to two ODEs:
\begin{align*}
A'(x)=&a A(x)^2,\\
B''(y)=&\frac{2 a B'(y)}{B(y)}+(2c_2+1)\frac{B'(y)^2}{B(y)},
\end{align*}
for some constant $a$.

The first ODE has general solution given by
\begin{align*}
A(x)=&\frac{1}{b-a x},
\end{align*}
while the second ODE has either a first integral:
\begin{equation}
B'(y)=\frac{B(y)^{1+2c_2}-2 a s}{s(1+2c_2)},\label{By}
\end{equation}
when $c_2\neq -\tfrac12$, 
or a first integral:
\begin{align*}
B'(y)=&2a\log(B(y))-s,
\end{align*}
when $c_2= -\tfrac12$. Thus, in the case of the separation $H=AB$, the solutions for $B$  are given in terms of quadratures. In these solutions $b$, $c$ and $s$ are constants.

Taking the appropriate values for $c_2$, this gives new solutions to the pEW and near-horizon geometry equations. 

\begin{example}[2D metric projective structures with skew-symmetric Ricci tensor and the reduced dKP equation] We now look closer at the pEW case, in which the value of the parameters are $c_1=1$, $c_2=-1$, $\la=0$. 

One class of solutions can be obtained by specialising to the case $c_2=-1$ in equation (\ref{By}). If $c_2=-1$ the general solution to (\ref{By}) is given implicitly by:
$$2as B(y)+\log\big(1-2as B(y)\big) -4a^2 s y+c=0,$$ 
with $c={\rm const}$. We note that such $B(y)$ is related to the Lambert function $W=W(z)$, which is defined implicitly as $z=W(z){\rm e}^{W(z)}$. Since the function $W\to W {\rm e}^W$ is not injective, one has more than one solutions to the equation  $z=W(z){\rm e}^{W(z)}$. If $z$ is real, there are two branches of the Lambert function: $W_0$ defined on $[-1/{\rm e},+\infty[$, with $W_0\geq -1$, and $W_{-1}$ defined on $[-1/{\rm e},0[$, with $W_{-1}\leq -1$. In terms of the branches of the Lambert function $W_\mu$, $\mu=0,-1$, the solution for $B$ reads:
$$B(y)=\frac{1+W_\mu(-{\rm e}^{c-1}{\rm e}^{4a^2s y})}{2as}.$$
We have:
\begin{proposition}\label{LambertW}
There are two branches, $\mu=0$ or $\mu=-1$, of 4-parameter $(a\neq 0,b,c,s\neq 0)$ Lorentzian generalised Ricci soliton pairs satisfying the pEW equations (i.e. generalised Ricci soliton equations with $c_1=1,c_2=-1,\la=0$) given by:
\[
\boxed{\begin{aligned}
g=& 2 \der x \der y+\frac{1+W_\mu(-{\rm e}^{c-1}{\rm e}^{4a^2s y})}{2as(b-ax)} \der x^2,
\\
X=& \frac{-2a W_\mu(-{\rm e}^{c-1}{\rm e}^{4a^2s y})}{(b-ax)(1+W_\mu(-{\rm e}^{c-1}{\rm e}^{4a^2s y}))} \der x.
\end{aligned}}
\]
These generalised Ricci solitons have $F \neq 0$. 
\end{proposition}

The Lorentzian generalised Ricci solitons described by Proposition \ref{LambertW} are particular examples of solutions to (\ref{nlpde2}) with $c_2=-1$. It turns out however, that for this particular value of $c_2$ the general solution of (\ref{nlpde2}) can be found. This is because if $c_2=-1$ equation (\ref{nlpde2}) becomes:  
\begin{equation}\label{rdkp}
0=-HH_{yy}-H_y^2+2 H_{xy}=(2 H_x-H H_y)_{y},
\end{equation}
and as such has an integral 
\begin{equation}\label{rdkpi}
2 H_x-H H_y=h(x).
\end{equation}
Surprisingly, equation (\ref{rdkp}) is the reduced dispersionless KP (dKP) equation, and its integral (\ref{rdkpi}) with $h(x)\equiv 0$ is the dispersionless KdV, also called as the Riemann-Hopf equation. Equation (\ref{rdkp}) arises in the study of 3-dimensional Lorentzian Einstein-Weyl equation with $S^1$ symmetry (see Section 3.1 of \cite{dKP}).  General solution to (\ref{rdkp}) depends implicitly on one arbitrary function of one variable, say $G(y)$. Using it, after making a hodographic transformation, we can write the corresponding Lorentzian pEW generalised Ricci soliton in the form:
\begin{align*}
g= 4 (x G(y)-1) \der x \der y,
\end{align*}
where $G(y)$ is one function of one variable, with 
\begin{align*}
X=&\frac{2 G(y)}{x G(y)-1}\der x,\\
F=&\frac{2 G'}{(x G(y)-1)^2} \der x \wedge \der y.
\end{align*}
We find that the Gauss curvature is
\begin{align*}
K=\frac{G'}{2(x G(y)-1)^3}.
\end{align*}
In general, this metric does not admit a local isometry unless the invariants $I_1$ and $I_2$ vanish, which is a differential constraint on $G$:
\[
G' G^2 (3G'''G'-5(G'')^2)=0.
\]
In this case solutions are given by 
\[
G(y)=a y+b \hspace{12pt} \mbox{or} \hspace{12pt} G(y)=\frac{3}{2}c^2\sqrt{\frac{-6 c}{y+d}}+e. 
\]
Both solutions have $F$ non-zero.
To summarise, we have
\begin{proposition}\label{reduceddKP}
In addition to the two 4 parameter families of examples given in Proposition \ref{LambertW}, there is also a family of Lorentzian generalised pEW Ricci solitons depending on one function of one variable given by:
\[
\boxed{\begin{aligned}
g=& 4 (x G(y)-1) \der x \der y,
\\
X=& \frac{2 G(y)}{x G(y)-1}\der x.
\end{aligned}}
\]
Furthermore, there is a 3-parameter family of examples which admit a Killing symmetry, given by $G(y)=\frac{3}{2}c^2\sqrt{\frac{-6 c}{y+d}}+e$, and another 2-parameter family of examples which admit a Killing symmetry, given by $G(y)=ay+b$.
\end{proposition}

The local obstructions to pEW generalised Ricci solitons derived in \cite{skewricci2D} all vanish for these examples.
\end{example}

\begin{example}[2D near-horizon geometry equation]
When $c_2=\frac{1}{2}$, we obtain a 2D non-static Lorentzian solution of the vacuum near-horizon geometry equation with no Killing symmetry. They again come form the separation $H(x,y)=A(x)B(y)$. If $c_2=\tfrac12$ the general solution to (\ref{By}) is given by:
$$B(y)=-\sqrt{2as}\tanh(\frac{\sqrt{a}(y+c)}{\sqrt{2s}}),$$
with $c={\rm const}$. This leads to the following 
\begin{proposition}
There is a 4-parameter $(a,b,c,s)$ family of Lorentzian generalised Ricci soliton pairs satisfying the vacuum near-horizon geometry equations given by:
\[
\boxed{\begin{aligned}
g=& 2 \der x \der y+\frac{\sqrt{2as}\tanh(\frac{\sqrt{a}(y+c)}{\sqrt{2s}})}{ax-b}\der x^2,
\\
X=& \frac{a~{\rm sech}^2(\frac{\sqrt{a}(y+c)}{\sqrt{2s}})}{2(ax-b)} \der x.
\end{aligned}}
\]
For this solutions $F\neq 0$. A computation shows that the Lorentzian metric $g$ admits no local Killing symmetry since the invariants $I_1$ and $I_2$ do not vanish. 
\end{proposition}
\end{example}

Finally, let us consider the case $c_1=1$ and the null vector $X$ given by 
\[
X=\frac{1}{2}E(x,y)H(x,y) \der x+E(x,y) \der y.
\]
Plugging this ansatz for $X$ into the $c_1=1$ generalised Ricci soliton equations, we find that the $\der y \der y$ component gives 
\[
E_y+E^2=0, 
\]
from which we obtain
\[
E(x,y)=\frac{1}{y+f(x)}.
\]
This gives the 1-form 
$$X=\frac{\der y+\tfrac12 H(x,y)\der x}{y+f(x)}.$$
The remaining equations to solve are equivalent to the following PDE:
\begin{align*}
2\big( c_2 H_{yy} (x,y)+2 \la\big) \big(y+f(x)\big)^2+H_{y}(x,y) f(x)+H_{y}(x,y) y+2 f'(x)-H(x,y)=0.
\end{align*}
A solution for $c_2 \neq 0,-\frac{1}{2}, -\frac{1}{4}$ is given by
$$\begin{aligned}
H(x,y)=&-\frac{4\la\big(y+f(x)\big)\big((1+2c_2)y-2c_2f(x)\big)}{(1+2c_2)(1+4c_2)}\\&+\big(y+f(x)\big)h(x)+\big(y+f(x)\big)^{-\tfrac{1}{2c_2}}j(x)+2f'(x),
\end{aligned}$$
where $h(x)$ and $j(x)$ are arbitrary functions. For this $H(x,y)$,  the Gauss curvature of the Lorentzian metric is
$$K=\tfrac12\Big(-\frac{8\la}{1+4c_2}+\frac{(1+2c_2)(y+f(x))^{-\tfrac{1+4c_2}{2c_2}}j(x)}{4c_2^2}\Big),$$
and the 2-form $F$ is:
$$F=\tfrac14\Big(\frac{8\la}{1+4c_2}+\frac{(1+2c_2)(y+f(x))^{-\tfrac{1+4c_2}{2c_2}}j(x)}{c_2}\Big)\der x\wedge\der y.$$
When $c_2=0$, we obtain the solution
$$\begin{aligned}
H(x,y)=\big(y+f(x)\big)\Big(h(x)-4\la\big(y+f(x)\big)\Big)+2f'(x),
\end{aligned}$$
where $h(x)$ is an arbitrary function. This gives a Lorentzian metric with constant Gauss curvature $K=-4\la$. Even here, if $\la\neq 0$, the corresponding generalised Ricci soliton is non-gradient as:
$$F=2\la \der x\wedge\der y.$$

When $c_2=-\frac{1}{2}$, we obtain the solution
$$\begin{aligned}
H&(x,y)=\\
&\big(y+f(x)\big)\Big(4\la y-4\la f(x)\big(\log(y+f(x))-1\big)+h(x)+j(x)\log(y+f(x))\Big)\\&+2f'(x),
\end{aligned}$$
where $h(x)$ and $j(x)$ are arbitrary functions. Here we obtain a Lorentzian metric with Gauss curvature
$$K=2\la+\frac{4\la y+j(x)}{2(y+f(x))},$$
and the 2-form $F$ is:
$$F=-\frac{4\la y+j(x)}{2(y+f(x))}\der x\wedge\der y.$$
When $c_2=-\frac{1}{4}$, we obtain the solution
$$\begin{aligned}
H&(x,y)=\big(y+f(x)\big)\\&\times\Big(h(x)+~y~\big(~8\la\big(\log(y+f(x))-1\big)+j(x)~\big)+f(x)~\big(~8\la\log(y+f(x))+j(x)~\big)\Big)\\&+2f'(x),
\end{aligned}$$
where $h(x)$ and $j(x)$ are arbitrary functions. Now the Gauss curvature of the Lorentzian metric is
$$K=4\la+8\la\log(y+f(x))+j(x),$$
and the 2-form $F$ is:
$$F=-\tfrac12(8\la\log(y+f(x))+j(x))\der x\wedge\der y.$$

To summarise, the corresponding $c_1=1$ generalised Ricci solitons with the other null $X$ are given by
\[
\boxed{\begin{aligned}
g=& 2 \der x \der y+H(x,y) \der x^2,
\\
X=& \frac{\der y+\tfrac12 H(x,y)\der x}{y+f(x)},
\end{aligned}}
\]
where
\begin{align*}
&\boxed{\begin{aligned}H(x,y)=&-\frac{4\la\big(y+f(x)\big)\big((1+2c_2)y-2c_2f(x)\big)}{(1+2c_2)(1+4c_2)}\\
&+\big(y+f(x)\big)h(x)+\big(y+f(x)\big)^{-\tfrac{1}{2c_2}}j(x)+2f'(x),\quad{\rm if} \quad c_2 \neq -\frac{1}{4}, -\frac{1}{2}, 0,
\end{aligned}}\\
&\boxed{\begin{aligned}
H(x,y)=\big(y+f(x)\big)\Big(h(x)-4\la\big(y+f(x)\big)\Big)+2f'(x),
\quad{\rm if} \quad c_2 =0,
\end{aligned}}\\
&\boxed{\begin{aligned}
H&(x,y)=\big(y+f(x)\big)\\&\times\Big(h(x)+\big(8\la\big(\log(y+f(x))-1\big)+j(x)\big)y+f(x)\big(8\la\log(y+f(x))+j(x)\big)\Big)\\&+2f'(x),
\quad{\rm if} \quad c_2 =-\frac{1}{4},
\end{aligned}}\\
&\boxed{\begin{aligned}
H(x,y)=&\big(y+f(x)\big)\Big(4\la y-4\la f(x)\big(\log(y+f(x))-1\big)+h(x)+j(x)\log(y+f(x))\Big)\\
&+2f'(x),
\quad{\rm if} \quad c_2 =-\frac{1}{2}.
\end{aligned}}
\end{align*}

\section{Summary and outlook}
Motivated by the method outlined in Section \ref{potential}, we obtain explicit examples of generalised Ricci solitons in 2 dimensions. We also obtain explicit examples in Lorentzian signature with $X$ null. The next step is to obtain higher dimensional generalised Ricci solitons. Following the work of \cite{KL0}, \cite{KL1}, \cite{KL2} and \cite{KL3} in constructing explicit examples of higher dimensional cohomogeneity-1 metrics satisfying the near-horizon geometry equations, we are able to get explicit generalised Ricci solitons in higher dimensions. However the presentation of this work will be left elsewhere. 

\section*{Acknowledgements}
Both authors would like to acknowledge Wojciech Kry\'nski for organising the workshop ``Geometry of Projective Structures and Differential Equations" in Warsaw, where this work was initialised. We also wish to thank Piotr Chru\'sciel, Jacek Jezierski and Paul Tod for helpful discussions. Special thanks are due to Gil Bor for his help in Section \ref{gil}.

\end{document}